\documentclass[generic,noinfoline]{imsart}

\RequirePackage[OT1]{fontenc}

\startlocaldefs
\usepackage{amsmath,amssymb,amsthm}
\usepackage{fancybox,fancyhdr,graphics,epsfig}
\usepackage[usenames,dvipsnames]{color}
\usepackage{bbm}
\usepackage{verbatim}
\usepackage{enumitem}
\usepackage{epstopdf}

\DeclareMathOperator{\diam}{diam}
\DeclareMathOperator{\vol}{Vol}

\DeclareMathOperator{\conv}{conv}

\def\R{\mathbb{R}}

\def\cC{\mathcal{C}}

\def\cL{\mathcal{L}}
\def\cM{\mathcal{M}}
\def\cN{\mathcal{N}}
\def\cP{\mathcal{P}}
\def\cU{\mathcal{U}}
\def\cS{\mathcal{S}}
\def\cX{\mathcal{X}}
\def\cY{\mathcal{Y}}

\def\cL{\mathcal{L}}

\def\d{\delta}

\def\oconv{\conv^{\circ}}

\newcommand{\E}{\mathbb{E}} 

\newcommand{\given}{\;|\;}
\newcommand{\mean}[1] {\E\left\{{#1}\right\}}
\newcommand{\meanx}[1] {\E\{{#1}\}}
\newcommand{\cmean}[2] {\E\left\{#1\given #2\right\}}
\newcommand{\ind}{\boldsymbol{\mathbbm{1}}} 

\newcommand{\var}[1]{\mathrm{Var}\param{{#1}}}

\newcommand{\set}[1]{\left\{#1\right\}}

\newcommand{\norm}[1]{\left\|#1\right\|}
\newcommand{\param}[1]{\left(#1\right)}
\newcommand{\abs}[1] {\left| {#1}\right|}

\newcommand{\prob}[1]{\mathbb{P}\left(#1\right)}

\newcommand{\eps}{\epsilon}

\newcommand{\by}{\mathbf{y}}
\newcommand{\bx}{\mathbf{x}}
\newcommand{\bv}{\mathbf{v}}

\def\hr{\hat{r}}
\newtheorem{lem}{Lemma}[section]
\newtheorem{thm}[lem]{Theorem}
\newtheorem{prop}[lem]{Proposition}
\newtheorem{cor}[lem]{Corollary}

\theoremstyle{definition}
\newtheorem{defn}[lem]{Definition}

\newcommand{\cech}{\v{C}ech }

\newcommand{\Rd}{\R^d}
\newcommand{\Rm}{\R^m}
\newcommand{\iid}{\mathrm{i.i.d.}}

\def\Nk{N_{k,n}}

\def\Sk{S_{k,n}}

\def\Fk{F_{k,n}}

\def\bk{\beta_{k,n}}

\def\Nkt{\widehat{N}_{k,n}}

\def\bkm{\beta_{k-1,n}}
\def\nabrn{\nabla_{r_n}}
\def\nabsn{\nabla_{s_n}}
\newcommand{\rnd}{r_n^d}
\newcommand{\rnm}{r_n^m}
\newcommand{\rnh}{\hat{r}_n}

\newcommand{\rnmk}{r_n^{mk}}

\newcommand{\factorm}{n^{k+1} \rnmk}
\def\bfactorm {n^{k+2}r_n^{m(k+1)}}
\newcommand{\ninf}{n\to\infty}
\newcommand{\pois}[1]{\mathrm{Poisson}\param{{#1}}}
\newcommand{\grn}{g_{r_n}}
\newcommand{\hrn}{h_{r_n}}
\newcommand{\fmax}{f_{\max}}
\newcommand{\fmin}{f_{\min}}
\newcommand{\limninf}{\lim_{\ninf}}

\newcommand{\dtv}[2]{d_{\mathrm{TV}}\param{{#1},{#2}}}
\newcommand{\bs}{\backslash}

\def\CC{\check{C}}

\numberwithin{equation}{section}

\def\d{{\rm d}}

\endlocaldefs

\begin{document}

\begin{frontmatter}

\title{The Topology of Probability Distributions on Manifolds}
\runtitle{The Topology of Probability Distributions on Manifolds}

\begin{aug}

\author{\fnms{Omer}  \snm{Bobrowski}\corref{}\ead[label=e1]{omer@math.duke.edu}}\thanksref{t1},
\and
\author{\fnms{Sayan} \snm{Mukherjee}\ead[label=e2]{sayan@stat.duke.edu}}\thanksref{t2}

\thankstext{t1}{OB was supported by DARPA: N66001-11-1-4002Sub\#8} 
 \thankstext{t2}{SM is pleased to acknowledge the support of NIH (Systems Biology): 5P50-GM081883, AFOSR: FA9550-10-1-0436, NSF CCF-1049290, and NSF DMS-1209155.}
 
\address{Department of Mathematics,\\ Duke University, Durham NC 27708\\ \printead{e1}}
  \address{Departments of Statistical Science,\\ Computer Science, and Mathematics\\
  	Institute for Genome Sciences \& Policy \\
	 Duke University, Durham NC 27708\\
          \printead{e2}}

\affiliation{Duke University}
\runauthor{Bobrowski and Mukherjee}
\end{aug}

\begin{abstract}
Let $\cP$ be a set of $n$ random points in $\R^d$, generated from a probability measure on a $m$-dimensional manifold $\cM\subset\R^d$. In this paper we study the homology of $\cU(\cP,r)$ --  the union of $d$-dimensional balls of radius $r$ around $\cP$, as $n\to \infty$, and $r\to 0$. In addition we study the critical points of $d_{\cP}$ --  the distance function from the set $\cP$. These two objects are known to be related via Morse theory. We present limit theorems for the Betti numbers of $\cU(\cP,r)$, as well as for number of critical points of index $k$ for $d_{\cP}$. Depending on how fast $r$ decays to zero as $n$ grows, these two objects exhibit different types of limiting behavior. In one particular case ($n r^m \ge C \log n$), we show that the Betti numbers of $\cU(\cP,r)$ perfectly recover the Betti numbers of the original manifold $\cM$, a result which is of significant interest in topological manifold learning.
\end{abstract}

\begin{keyword}[class=AMS]
\kwd[Primary ]{60D05, 60F15, 60G55}
\kwd[; secondary ]{55U10}
\end{keyword}

\begin{keyword}
\kwd{Random complexes}
\kwd{Point process}
\kwd{Random Betti numbers}
\kwd{Stochastic topology}
\end{keyword}

\end{frontmatter}

\section{Introduction}
The incorporation of geometric and topological concepts for statistical inference is at the heart of spatial point process models, manifold learning, and topological data analysis. The motivating principle behind manifold learning is using low dimensional geometric summaries of the data for statistical inference \cite{AswaniBickelTomlin,BelkinNiyLap,ChenMuller, RoheChatterjeeYu,WuMukZhou}.  In topological data analysis, topological summaries of data are used to infer or extract underlying structure in the data \cite{TaylorWorsley2,Genoveseetal,MischaikowWanner, Adleretal, TaylorWorsley1}. In the analysis of spatial point processes, limiting distributions of integral-geometric quantities such as area and boundary length \cite{Diggle2003,MW2003,Ripley,ms05}, Euler characteristic of patterns of discs centered at random points \cite{ms05,stoyan_stochastic_1987},  and the Palm mean (the mean number of pairs of points within a radius $r$) \cite{Diggle2003, MW2003,stoyan_stochastic_1987,Ripley} have been used to characterize parameters of point processes, see \cite{ms05} for a short review.

A basic research topic in both manifold learning and topological data analysis is understanding the distribution of geometric and topological quantities generated by a stochastic process. In this paper we consider the standard model in topological data analysis and manifold learning -- the stochastic process is a random sample of points $\cP$ drawn from a distribution supported on a compact $m$-dimensional manifold $\cM$, embedded in $\R^d$. In both geometric and topological data analysis, understanding the local neighborhood structure of the data is important. Thus, a central parameter in any analysis is the size $r$ (radius) of a local neighborhood and how this radius scales with the number of observations.

We study two different, yet related, objects. The first object is the union of the $r$-balls around the random sample, denoted by $\cU(\cP,r)$. For this object, we wish to study its homology, and in particular its Betti numbers. Briefly, the Betti numbers are topological invariants measuring the number of components and holes of different dimensions.
Equivalently, all the results in this paper can be phrased in terms of the \cech complex $\CC(\cP,r)$. A simplicial complex is a collection of vertices, edges, triangles, and higher dimensional faces, and can be thought of as a generalization of a graph. The \cech complex $\CC(\cP,r)$ is simplicial complex where each $k$-dimensional face corresponds to an intersection of $k+1$ balls in $\cU(\cP,r)$ (see Definition \ref{def:cech_complex}). By the famous `Nerve Lemma' (cf. \cite{borsuk_imbedding_1948}), $\cU(\cP,r)$ has the same homology as $\CC(\cP,r)$.
The second object of study is the distance function from the set $\cP$, denoted by $d_{\cP}$, and its critical points. The connection between these two objects is given by Morse theory, which will be reviewed later. In a nutshell, Morse theory describes how critical points of a given function either create or destroy homology elements (connected components and holes) of sublevel sets of that function.

We characterize the limit distribution of the number of critical points of $d_{\cP}$, as well as the Betti numbers of $\cU(\cP,r)$.
Similarly to many phenomena in random geometric graphs  as well as random geometric complexes in Euclidean spaces \cite{kahle_random_2011,kahle_limit_2010,Luna:2009, penrose_random_2003},
the qualitative behavior of these distributions falls into three main categories based on how the radius $r$ scales with the number of samples $n$. This behavior is determined by the term $nr^m$, where $m$ is the intrinsic dimension of the manifold. This term can be thought of as the expected number of points in a ball of radius $r$. We call the different categories -- the sub-critical ($nr^m\to0$), critical ($nr^m\to \lambda$) and super-critical ($nr^m\to\infty$) regimes. The union $\cU(\cP,r)$ exhibits very different limiting behavior in each of these three regimes. In the sub-critical regime,  $\cU(\cP,r)$ is very sparse and consists of many small particles, with very few holes. In the critical regime, $\cU(\cP,r)$ has $O(n)$ components as well as holes of any dimension $k<m$. From the manifold learning perspective, the most interesting regime would be the super-critical. One of the main result in this paper (see Theorem \ref{thm:lim_betti_sup}) states that if we properly choose the radius $r$ within the super-critical regime, the homology of the random space $\cU(\cP,r)$ perfectly recovers the homology of the original manifold $\cM$. This result extends the work in \cite{niyogi_finding_2008}  for a large family of distributions on $\cM$, requires much weaker assumptions on the geometry of the manifold, and is proved to happen almost surely.

The study of critical points for the distance function provides additional insights on the behavior of $\cU(\cP,r)$ via Morse theory, we return to
this later in the paper. While Betti numbers deal with `holes' which are typically determined by global phenomena, the structure of critical points is mostly local in nature.
Thus, we are able to derive precise results for critical points even in cases where we do not have precise analysis of Betti numbers.
One of the most interesting consequence of the critical point analysis in this paper relates to the Euler characteristic of $\cU(\cP,r)$. 
One way to think about the Euler characteristic of a topological spaces $\cS$ is as integer ``summary" of the Betti numbers given by $\chi(\cS) = \sum_{k} (-1)^k \beta_k(\cS)$. Morse theory enables us to compute $\chi(\cU(\cP,r))$  using the critical points of the distance function $d_{\cP}$ (see Section \ref{sec:crit_range}). This computation may provide important insights on the behavior of the Betti numbers in the critical regime. We note that the equivalent result for Euclidean spaces appeared in \cite{bobrowski_distance_2011}.

In topological data analysis there has been work on understanding properties of random abstract simplicial complexes generated from stochastic processes \cite{adler_persistent_2010,ABW,aronshtam_vanishing_2010,bobrowski_distance_2011,kahle_topology_2009, kahle_random_2011, kahle_limit_2010, linial2006homological, penrose_random_2003,penrose_limit_2011} and non-asymptotic bounds on  the convergence or consistency of topological summaries  as the number of points increase \cite{BenMukWang2012,bubenik_statistical_2009,bubenik_statistical_2007,ChaCohLie2009,chung_persistence_2009,niyogi_finding_2008,niyogi_topological_2011}. The central idea in  these papers has been to study statistical properties of topological summaries of point cloud data. There has also been an effort to characterize the topology of a distribution (for example a noise model) \cite{adler_persistent_2010,ABW,kahle_limit_2010}. Specifically, the results in our paper adapt the results in \cite{bobrowski_distance_2011,kahle_random_2011,kahle_limit_2010} from the setting of a distribution in Euclidean space to one supported on a manifold.

There is a natural connection of the results in this paper with results in point processes, specifically random set models such as the Poisson-Boolean model \cite{MeesterRoy96}. The stochastic process we study in this paper is an example of a random set model -- stochastic models that place probability distributions on regions or geometric objects \cite{Matheron1975,Molchanov2005}. Classically, people have studied limiting distributions of quantities such as volume, surface area, integral of mean curvature and Euler characteristic generated from the random set model. Initial studies examined  second order statistics, summaries of observations that measure position or interaction among points, such as the distribution function of nearest neighbors, the spherical contact distribution function, and a variety of other summaries such as Ripley's K-function, the L-function and the pair correlation function, see \cite{Diggle2003, MW2003,Ripley, stoyan_stochastic_1987}. It is known that there are limitations in only using second order statistics since one can state different point processes that have the same second order statistics \cite{BS1984}. In the spatial statistics literature our work is related to the use of morphological functions for point processes where a ball of radius $r$ is placed around each point sampled from the point process and the topology or morphology of the union of these balls is studied.  Our results are also related to ideas in the statistics and statistical physics of random fields, see \cite{adler_random_2007, AuffingerBenArous,bobrowski2012euler, worsley_boundary_1995, worsley_estimating_1995}, a random process on a manifold can be thought of as an approximation of excursion sets of Gaussian random fields or energy landscapes.

The paper is structured as follows. In Section \ref{sec:topology} we give a brief introduction to the topological objects we study in this paper.   In Section \ref{sec:setup} we state the probability model and define the relevant topological and geometric quantities of study. In sections \ref{sec:results} and \ref{proofs} we state our main results and proofs, respectively.

\section{Topological Ingredients}
\label{sec:topology}
In this paper we study two topological objects generated from a finite random point cloud $\cP \subset\R^d$ (a set of points in $\R^d$).
\begin{enumerate}
\item Given the set $\cP$ we define
\begin{equation}\label{eq:def_union_balls}
	\cU(\cP, \eps) := \bigcup_{p\in \cP} B_\eps(p),
\end{equation}
where $B_\eps(p)$ is a $d$-dimensional ball of radius $\eps$ centered at $p$. Our interest in this paper is in characterizing the \emph{homology} --
in particular the \emph{Betti numbers} of this space, i.e.\! the number of components, holes, and other types of voids in the space.
\item We define the distance function from $\cP$ as
\begin{equation}\label{eq:def_dist_fn}
	d_{\cP}(x) := \min_{p\in \cP}\norm{x-p},\quad x\in \Rd.
\end{equation}
As a real valued function, $d_{\cP}:\R^d\to \R$ might have critical points of different types (i.e.\! minimum, maximum and saddle points). We would like to study the amount and type of these points.
\end{enumerate}

In this section we give a brief introduction to the topological concepts behind these two objects.
Observe that  the  sublevel sets of the distance function are
\[
	d_{\cP}^{-1}((-\infty,r]) := \set{x\in\R^d : d_{\cP}(x) \le r} =  \cU(\cP,r).
\]
Morse theory, discussed later in this section, describes the interplay between critical points of a function and the homology of its sublevel sets, and hence provides the link between our two objects of study.

\subsection{Homology and Betti Numbers}

Let $X$ be a topological space. The $k$-th Betti number of $X$, denoted by $\beta_k(X)$ is the rank of $H_k(X)$ -- the $k$-th homology group of $X$.
This definition assumes that the reader has a basic grounding in algebraic topology. Otherwise, the reader should be willing to accept a definition of $\beta_k(X)$  as the number of  $k$-dimensional `cycles' or `holes' in $X$, where a $k$-dimensional hole can be thought of as anything that can be continuously transformed into the boundary of a  $(k+1)$-dimensional shape. The zeroth Betti number,
$\beta_0(X)$, is merely the number of connected components in $X$. For example, the $2$-dimensional torus $T^2$ has a single connected component, two non-trivial $1$-cycles, and a $2$-dimensional void. Thus, we have that $\beta_0(T^2) = 1,\ \beta_1(T^2) = 2,$ and $\beta_2(T^2)=1$. Formal definitions of homology groups and Betti numbers can be found in
\cite{hatcher_algebraic_2002, munkres1984elements}.

\subsection{Critical Points of the Distance Function}\label{sec:crit_pts}

The classical definition of critical points using calculus is as follows. Let $f:\R^d\to\R$ be a $C^2$ function. A point $c\in \R$ is called a \textit{critical point} of $f$ if $\nabla f (c) =0$, and the real number $f(c)$ is called a \textit{critical value} of $f$. A critical point $c$ is called \textit{non-degenerate} if the Hessian matrix $H_f(c)$ is non-singular. In that case, the \textit{Morse index} of $f$ at $c$, denoted by $\mu(c)$ is the number of negative eigenvalues of $H_f(c)$. A $C^2$ function $f$ is a \textit{Morse function} if all its critical points are non-degenerate, and its critical levels are distinct.

Note, the distance function $d_{\cP}$ is not everywhere differentiable, therefore the definition above does not apply. However, following \cite{gershkovich_morse_1997}, one can still define a notion of non-degenerate critical points for the distance function, as well as their Morse index.
Extending Morse theory to functions that are non-smooth has been developed for a variety of applications  \cite{baryshnikov2013min, Bryzgalova78,gershkovich_morse_1997, Matov82}. The class of functions studied in these papers have been the minima (or maxima) of a functional
and called `min-type' functions. In this section, we specialize those results to the case of the distance function.

We start with the local (and global) minima of $d_{\cP}$, the points of
$\cP$ where $d_{\cP} = 0$, and call these
critical points with index $0$. For higher indices,
 we have the following definition.

\begin{defn}\label{def:crit_pts}
A point $c\in\R^d$ is \emph{a critical point of  index $k$ of $d_{\cP}$}, where $1 \le k \le d$, if there exists a subset $\cY$ of $k+1$ points in $\cP$ such that:
\begin{enumerate}
\item $\forall y\in \cY:  d_{\cP}(c) = \norm{c-y} $, and, $\forall p\in \cP \backslash \cY$ we have $\norm{c-p} > d_{\cP}(p)$.
\item The points in $\cY$ are in general position (i.e.\! the $k+1$ points
of $\cY$ do not lie in a $(k-1)$-dimensional affine space).
\item $c \in \oconv(\cY)$,
where $\oconv(\cY)$ is the interior of the convex hull of $\cY$ (an open $k$-simplex in this case).
\end{enumerate}
\end{defn}
The first condition implies that $d_{\cP} \equiv d_{\cY}$ in a small neighborhood of $c$.
The second condition implies that the points in $\cY$ lie on a unique $(k-1)$- dimensional sphere. We shall use the following notation:
\begin{align}
S(\cY) &= \textrm{The unique $(k-1)$-dimensional sphere containing $\cY$},\\
C(\cY) &= \textrm{The center of $S(\cY)$ in $\R^d$}, \label{eq:def_C}\\
R(\cY) &= \textrm{The radius of $S(\cY)$} , \label{eq:def_R}\\
B(\cY) &= \textrm{The open ball in $\R^d$ with radius $R(\cY)$ centered at $C(\cY)$}.\label{eq:def_B}
\end{align}
Note that $S(\cY)$ is a $(k-1)$-dimensional sphere, whereas $B(\cY)$ is a $d$-dimensional ball. Obviously, $S(\cY) \subset B(\cY)$, but unless $k=d$, $S$ is
{\it not} the boundary of $B$.
Since the critical point $c$ in Definition \ref{def:crit_pts} is equidistant from all the points in $\cY$, we have that $c=C(\cY)$. Thus, we say that $c$ is the unique index $k$ critical point \textit{generated} by the $k+1$ points in $\cY$.
The last statement can be rephrased as follows:
\begin{lem} \label{lem:gen_crit_point}
A subset $\cY\subset \cP$ of $k+1$ points in general position generates an index $k$ critical point if, and only if, the following two conditions hold:
\begin{enumerate}[label=\bf{CP{\arabic*}}]
\item \label{cp1} $\quad C(\cY) \in \oconv(\cY)$,
\item \label{cp2} $\quad \cP \cap {B(\cY)}= \emptyset$.
\end{enumerate}
Furthermore, the critical point is $C(\cY)$ and the critical value is $R(\cY)$.
\end{lem}

Figure \ref{fig:crit_pts} depicts the generation of an index $2$ critical point in $\R^2$ by  subsets of $3$ points.
We shall also be interested in `local' critical points, points where $d_{\cP}(c)\le \eps$. This adds a third condition,
\begin{enumerate}[label=\bf{CP\arabic*}]
\setcounter{enumi}{2}
\item \label{cp3.eps} $\quad R(\cY) \le \eps$.
\end{enumerate}
\begin{figure}[h]
\centering
 \includegraphics[scale=0.3]{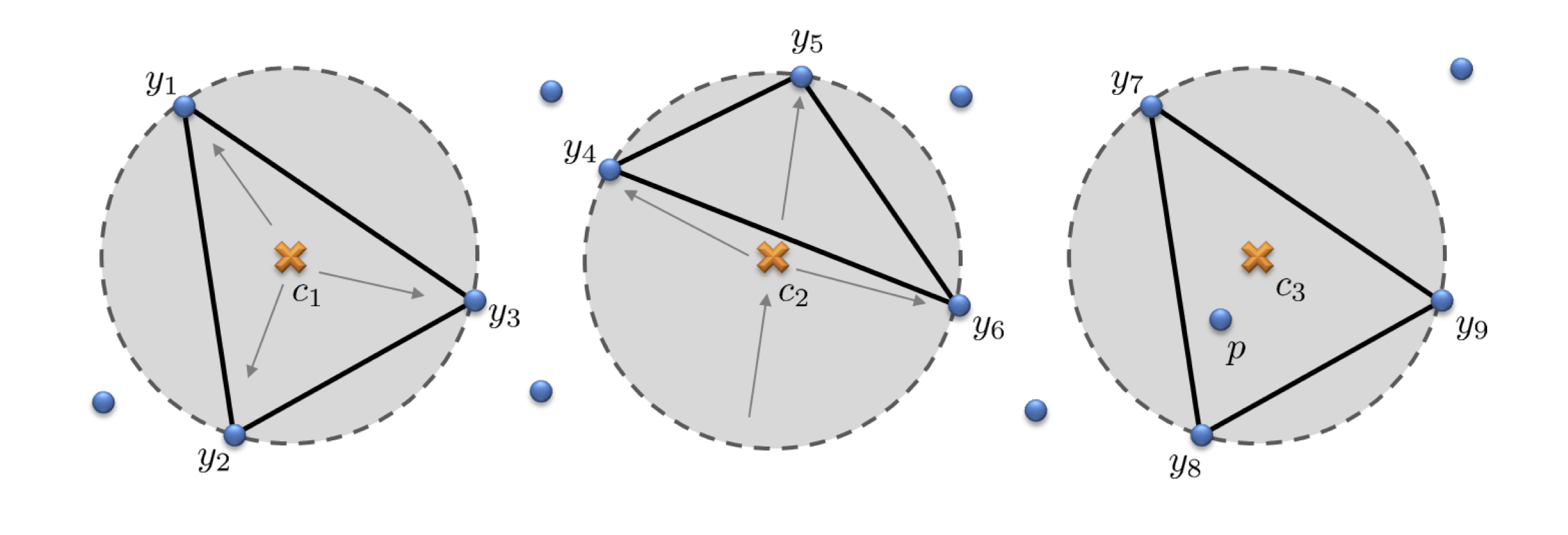}
\caption{\label{fig:crit_pts} Generating a critical point of index $2$ in $\R^2$, a maximum point.
The small blue disks are the points of  $\cP$. We examine three subsets of $\cP$: $\cY_1 = \set{y_1,y_2,y_3}$, $\cY_2 = \set{y_4,y_5,y_6}$, and $\cY_3 = \set{y_7,y_8,y_9}$. $S(\cY_i)$ are the dashed circles, whose centers are $C(\cY_i) = c_i$. The shaded balls are $B(\cY_i)$, and the interior of the triangles are $\oconv(\cY_i)$.
(1) We see that both  $C(\cY_1) \in \oconv(\cY_1)$ \eqref{cp1} and $\cP\cap B(\cY_1) = \emptyset$ \eqref{cp2}. Hence $c_1$ is a critical point of index $2$.
(2) $C(\cY_2) \not\in \oconv(\cY_2)$, which means that \eqref{cp1} does not hold, and therefore $c_2$ is not a critical point (as can be observed from the flow arrows).
(3) $C(\cY_3) \in \oconv(\cY_3)$, so \eqref{cp1} holds. However, we have $\cP\cap B(\cY_3) = \set{p}$, so \eqref{cp2} does not hold, and therefore $c_3$ is also not a critical point. Note that in a small neighborhood of $c_3$ we have $d_{\cP} \equiv d_{\set{p}}$, completely ignoring the existence of $\cY_3$.}
\end{figure}

The following indicator functions, related to CP1--CP3, will appear often.
\begin{defn}\label{def:ind_fns}
Using the notation above,
\begin{align}
h^c(\cY) & := \ind\set {C(\cY) \in \oconv(\cY)}\qquad\qquad\quad\ \ (\ref{cp1}) \label{eq:def_h} \\
h^c_\eps(\cY) & := h^c(\cY) \ind_{[0,\eps]} (R(\cY)) \qquad\qquad \qquad\qquad (\ref{cp1}+\ref{cp3})\label{eq:def_h_eps}\\
g^c_\eps(\cY,\cP) & := h^c_\eps(\cY) \ind\set{ \cP \cap {B(\cY)} = \emptyset} \ \qquad (\ref{cp1}+\ref{cp2}+\ref{cp3}) \label{eq:def_g_eps}
\end{align}
\end{defn}

\subsection{Morse Theory}\label{sec:morse}

The study of homology is strongly connected to the study of critical points of real valued functions. The link between them is called Morse theory, and we shall describe it here briefly. For a deeper introduction, we refer the reader to \cite{milnor_morse_1963}.

Let $\cM$ be a smooth manifold embedded in $\R^d$, and let $f:\cM\to\R$ be a Morse function (see Section \ref{sec:crit_pts}).

The main idea of Morse theory is as follows. Suppose that $\cM$ is a closed manifold (a compact manifold without a boundary), and let $f:\cM\to \R$ be a Morse function.
Denote
\[
\cM_{\rho} := f^{-1}((-\infty,\rho]) = \set{x\in \cM : f(x) \le \rho}\subset \cM
\] (sublevel sets of $f$).
If there are no critical levels in $(a,b]$, then $\cM_a$ and $\cM_b$ are \emph{homotopy equivalent}, and in particular have the same homology.
Next, suppose that $c$ is a critical point of $f$ with Morse index $k$, and let $v=f(c)$ be the critical value at $c$. Then the homology of $\cM_{\rho}$ changes at $v$ in the following way. For a small enough $\eps$ we have that the homology of $\cM_{v+\eps}$ is obtained from the homology of $\cM_{v-\eps}$ by either adding  a generator to $H_k$ (increasing $\beta_k$ by one) or terminating a generator of $H_{k-1}$ (decreasing $\beta_{k-1}$ by one). In other words, as we pass a critical level, either a new $k$-dimensional hole is formed, or an existing $(k-1)$-dimensional hole is terminated (filled up).

Note, that while classical Morse theory deals with Morse functions  (and in particular, $C^2$)  on compact manifolds, its extension for min-type functions presented in \cite{gershkovich_morse_1997} enables us to apply these concepts to the distance function $d_{\cP}$ as well.

\subsection{\cech Complexes and the Nerve Lemma}\label{sec:cech}

The \cech complex generated by a set of points $\cP$ is a simplicial complex, made up of vertices, edges, triangles and higher dimensional faces.
While its general definition is quite broad, and uses intersections of arbitrary nice sets, the following special case using intersection of Euclidean balls will be sufficient for our analysis.

\begin{defn}[\cech complex]\label{def:cech_complex}
Let $\cP = \set{x_1,x_2,\ldots}$ be a collection of points in $\R^d$, and let $\eps>0$. The \cech complex $\CC(\cP, \eps)$ is constructed as follows:
\begin{enumerate}
\item The $0$-simplices (vertices) are the points in $\cP$.
\item An $n$-simplex $[x_{i_0},\ldots,x_{i_n}]$ is in $\CC(\cP,\eps)$ if $\bigcap_{k=0}^{n} {B_{\eps}(x_{i_k})} \ne \emptyset$.
\end{enumerate}
\end{defn}
Figure \ref{fig:cech} depicts a simple example of a \cech complex in $\R^2$.
\begin{figure}[h!]
\centering
  \includegraphics[scale=0.4]{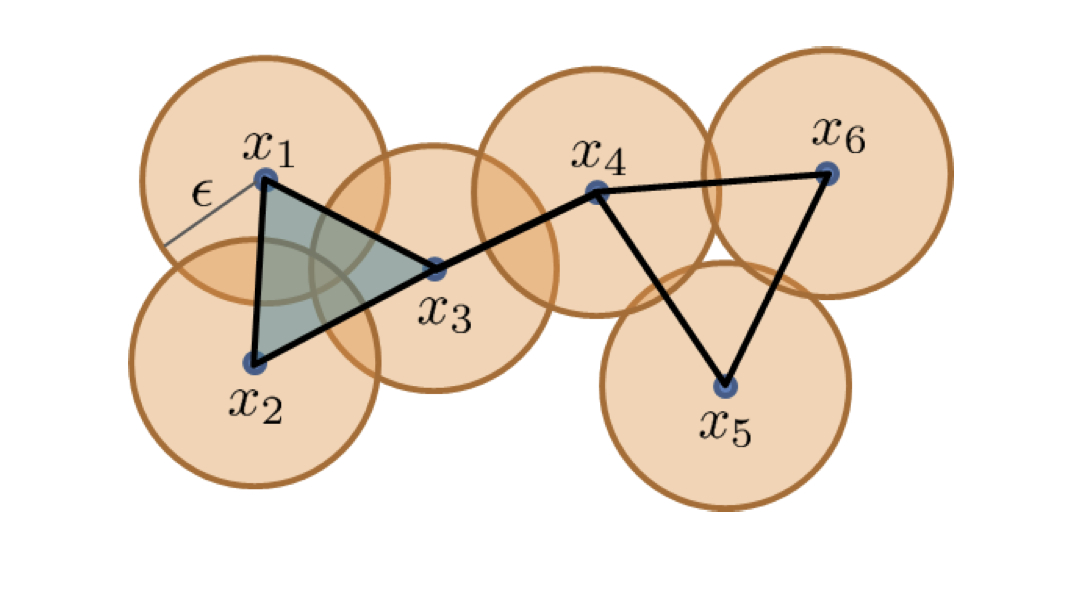}
\caption{\label{fig:cech} The \cech complex $\CC(\cP,\eps)$, for $\cP = \set{x_1,\ldots,x_6}\subset \R^2$, and some  $\eps$. The complex contains 6 vertices, 7 edges, and a single 2-dimensional face.}
\end{figure}
An important result, known as the `Nerve Lemma', links the \cech complex
 $\CC(\cP,\eps)$ and the neighborhood set  $\cU(\cP, \eps)$,
and states that they are  homotopy equivalent, and in particular they have the same homology groups  (cf. \cite{borsuk_imbedding_1948}). Thus, for example,
they have the same Betti numbers.

Our interest in the \cech complex is twofold. Firstly, the \cech complex is a high-dimensional analogue of a geometric graph. The study of random geometric graphs is well established (cf.\! \cite{penrose_random_2003}). However, the study of higher dimensional geometric complexes is at its early stages.   Secondly, many of the proofs in this paper are combinatorial in nature. Hence, it is usually easier to examine the \cech complex $\CC(\cP, \eps)$, rather than the geometric structure $\cU(\cP,\eps)$.

\section{Model Specification and Relevant Definitions}
\label{sec:setup}

In this section we specify the stochastic process on a manifold that generates the point sample and topological summaries we will characterize.

The point processes we examine in this paper live in $\R^d$ and are supported on a $m$-dimensional manifold $\cM\subset\R^d$ ($m<d$). Throughout this paper we assume that $\cM$ is closed (i.e.\! compact and without a boundary) and smooth.

Let $\cM$ be such a manifold, and let $f:\cM\to\R$ be a probability density function on $\cM$, which we assume to be bounded and measurable. If $X$ is a random variable in $\R^d$ with density $f$, then for every $A\subset \R^d$
\[	
	F(A) := \prob{X\in A} = \int_{A\cap \cM} f(x) \, \d x,
\]
where $\d x$ is the volume form on $\cM$.

We consider two models for generating point clouds on the manifold $\cM$: \begin{enumerate}
\item[(1)] \emph{Random sample: } $n$ points are drawn $\cX_n = \set{X_1, X_2, \ldots, X_n} \stackrel{iid}{\sim} f$, \\
\item[(2)] \emph{Poisson process: } the points are drawn from a spatial Poisson process with intensity function  $\lambda_n := n f$.
The  spatial Poisson process has the following two properties:
\begin{enumerate}
\item[(a)] For every region $A\subset \cM$, the number of points in the region $N_A := \abs{\cP_n \cap A}$ is distributed as a Poisson random variable
\[
	N_A \sim \pois{n F(A)};
\]
\item[(b)] For every $A,B\subset \cM$ such that $A\cap B = \emptyset$, the random variables $N_A$ and $N_B$ are independent.
\end{enumerate}
\end{enumerate}

These two models behave very similarly. The main difference is that the number of points in $\cX_n$ is exactly $n$, while the number of points in $\cP_n$ is distributed $\pois{n}$. Since the Poisson process has computational advantages, we will present all the results and proofs in this paper in terms of $\cP_n$.
However, the reader should keep in mind that the results also apply to samples generated by the first model ($\cX_n$),  with some minor adjustments. For a full analysis of the critical points in the Euclidean case for both models, see \cite{bobrowski2012thesis}.

The stochastic objects we study in this paper are the union $\cU(\cP_n, \eps)$ (defined in \eqref{eq:def_union_balls}), and the distance function $d_{\cP_n}$ (defined in \eqref{eq:def_dist_fn}). The random variables we examine are the following.
Let $r_n$ be a sequence of positive numbers, and define
\begin{equation}
	\bk := \beta_k(\cU(\cP_n,r_n)),
\end{equation}
to be the $k$-th Betti number of $\cU(\cP_n,r_n)$, for $0\le k \le d-1$. The values $\bk$ form a set of well defined integer random variables.

For $0\le k \le d$, denote by $\cC_{k,n}$ the set of critical points with index $k$ of the
distance function $d_{\cP_n}$.
Let $r_n$ be positive , and define the set of `local' critical points as
\begin{equation} \label{eq:ckl_def}
\cC_{k,n}^L := \set{c\in\cC_{k,n} : d_{\cP_n}(c) < r_n} = \cC_{k,n}\cap \cU(\cP_n,r_n);
\end{equation}
and its size as
\begin{equation} \label{eq:def_nk}
N_{k,n} := \abs{\cC_{k,n}^L} .
\end{equation}
The values $\Nk$ also form a set of integer valued random variables. From the discussion in Section \ref{sec:morse} we know that there is a strong connection between the set of values $\set{\bk}_{k=0}^{d-1}$ and $\set{\Nk}_{k=0}^{d}$.
We are interested in studying the limiting behavior of these two sets of random variables, as $n\to\infty$, and $r_n\to 0$.

\section{Results}
\label{sec:results}

In this section we present limit theorems for the random variables $\bk$ and $\Nk$, as $n\to\infty$, and $r_n\to 0$.
Similarly to the results presented in \cite{bobrowski_distance_2011, kahle_random_2011}, the limiting behavior splits into three main regimes. In \cite{bobrowski_distance_2011, kahle_random_2011} the term controlling the behavior is $n\rnd$, where $d$ is the ambient dimension. This value can be thought of as representing the expected number of points occupying a ball of radius $r_n$. Generating samples from a $m$-dimensional manifold (rather than the entire $d$-dimensional space) changes the controlling term to be $n\rnm$. This new term can be thought of as the expected number of points occupying a geodesic ball of radius $r_n$ on the manifold. We name the different regimes the \emph{sub-critical} ($n\rnm\to 0$), the \emph{critical} ($n\rnm\to\lambda$), and the super-critical ($n\rnm\to\infty$).
In this section we will present limit theorems for each of these regimes separately. First, however, we present a few statements common to all regimes.


The index $0$ critical points (minima) of $d_{\cP_n}$ are merely the points in $\cP_n$. Therefore, $N_{0,n} = \abs{\cP_n} \sim \pois{n}$, so our focus is on  the higher indexes critical points.

Next, note that if the radius $r_n$ is small enough, one can show that $\cU(\cP_n,r_n)$ can be continuously transformed into a subset $\cM'$ of $\cM$ (by a `deformation retract'), and this implies that $\cU(\cP_n,r_n)$ has the same homology as $\cM'$. Since $\cM$ is $m$-dimensional, $\beta_k(\cM) = 0$ for every $k>m$, and the same goes for every subset of $\cM$.
In addition, except for the coverage regime (see Section \ref{sec:results_supcrit}), $\cM'$ is a union of  strict subsets of the connected components of $\cM$, and thus must have $\beta_m(\cM') = 0$ as well. Therefore, we have that $\bk$ = 0 for every $k\ge m$. By Morse theory, this also implies that $\Nk=0$ for every $k>m$.
 The results we present in the following sections therefore focus on  $\beta_{0,n},\ldots,\beta_{m-1,n}$ and $N_{1,n},\ldots,N_{m,n}$ only.

\subsection{The Sub-Critical Range $(n\rnm \to 0)$}

In this regime, the radius $r_n$ goes to zero so fast, that the average number of points in a ball of radius $r_n$ goes to zero. Hence, it is very unlikely for points to connect, and $\cU(\cP_n,r_n)$ is very sparse. Consequently this phase is sometimes called the `dust' phase. We shall see that in this case $\beta_{0,n}$ is dominating all the other Betti numbers, which appear in a descending order of magnitudes.

\begin{thm}[Limit mean and variance]\label{thm:lim_mean_var_subcrit}
If $n\rnm\to 0$, then
\begin{enumerate}
\item For  $1\le k \le m-1$,
\[
\lim_{\ninf} \frac{\mean{\bk}}{\bfactorm} = \lim_{\ninf} \frac{\var{\bk}}{\bfactorm}  = \mu^b_{k},
\]
and
\[
	\limninf n^{-1}\mean{\beta_{0,n}} = 1.
\]
\item For $1 \le k \le m$,
\[
\lim_{\ninf} \frac{\mean{\Nk}}{\factorm} = \lim_{\ninf}\frac{\var{\Nk}}{\factorm} =  \mu^c_k.
\]
\end{enumerate}
where
\begin{align*}
\mu^b_{k} &= \frac{1}{(k+2)!}\int_{\cM}f^{k+2}(x)dx \int_{(\Rm)^{k+1}} h^b_1(0,\by)d\by,\\
\mu_k^c &= \frac{1}{(k+1)!}\int_{\cM}f^{k+1}(x)dx \int_{(\Rm)^k} h_1^c(0,\by)d\by.
\end{align*}
 The function $h^b_\eps$ is an indicator function on subsets $\cY$ of size $k+2$, testing that a subset forms a non-trivial $k$-cycle, i.e.
 \begin{equation}\label{eq:def_hb}
 h^b_{\eps}(\cY) := \ind\set{\beta_k(\cU(\cY,\eps)) = 1},
 \end{equation}
The function $h_\eps^c$ is defined in \eqref{eq:def_h_eps}.

Finally, we note that for $\by = (y_1,\ldots,y_{k+1}) \in (\R^d)^{k+1}$,   $h_\eps^b(0,\by) := h_\eps^b(0,y_1,\ldots,y_{k+1})$,
and for $\by = (y_1,\ldots,y_{k}) \in (\R^d)^{k}$,   $h_\eps^c(0,\by) := h_\eps^c(0,y_1,\ldots,y_{k})$
\end{thm}

Note that these results  are analogous to the limits in the Euclidean case, presented in \cite{kahle_limit_2010} (for the Betti numbers) and \cite{bobrowski_distance_2011} (for the critical points). 
In general, as is common for results of this nature,
it is difficult to express the integral formulae above in a more transparent form.
Some numerics as well as special cases evaluations are presented in \cite{bobrowski_distance_2011}.


Since $n\rnm\to 0$, the comparison between the different limits yields the following picture,
\[
  \begin{array}{ccccccccc}
    \mean{N_{0,n}} &\gg &\mean{N_{1,n}} & \gg & \mean{N_{2,n}} & \gg & \mean{N_{3,n}}& \gg  \cdots  \gg  & \mean{N_{m,n}}\\
	{\rotatebox{90}{\scalebox{1}[1]{$\approx$}}} &  & &  & {\rotatebox{90}{\scalebox{1}[1]{$\approx$}}} &  &{\rotatebox{90}{\scalebox{1}[1]{$\approx$}}} & & {\rotatebox{90}{\scalebox{1}[1]{$\approx$}}}\\
       \mean{\beta_{0,n}}&& \gg  & & \mean{\beta_{1,n}} & \gg & \mean{\beta_{2,n}} & \gg \cdots \gg & \mean{\beta_{m-1,n}} ,
  \end{array}
\]
where by $a_n\approx b_n$ we mean that $a_n/b_n \to c\in(0,\infty)$ and by $a_n \gg b_n$ we mean that $a_n/b_n \to \infty$. This diagram implies that in the sub-critical phase the dominating Betti number is $\beta_0$. It is significantly less likely to observe any cycle, and it becomes less likely as the cycle dimension increases. In other words, $\cU(\cP_n,r_n)$ consists mostly of small disconnected particles, with relatively few holes.

Note that the limit of the term $\factorm$ can be either zero, infinity, or anything in between. For each of these cases, the limiting distribution of either $\bkm$ or $\Nk$ is completely different. The results for the number of critical points are as follows.

\begin{thm}[Limit distribution]\label{thm:dist_subcrit}
Let $n\rnm\to 0$, and  $1 \le k \le m$,
\begin{enumerate}
\item  If $\lim_{n\to\infty} n^{k+1}r_n^{k} = 0$, then
\[
\Nk \xrightarrow{L^2} 0.
\]
If, in addition, $\sum_{n=1}^\infty \factorm < \infty$, then
\[
\Nk \xrightarrow{a.s.} 0.
\]
\item If $\lim_{n\to\infty} n^{k+1}r_n^{mk} = \alpha \in (0,\infty)$, then
\[
\Nk \xrightarrow{\cL} \pois{{\alpha \mu^c_k}}.
\]
\item If $\lim_{n\to\infty} n^{k+1}r_n^{mk} = \infty$, then
\[
\frac{\Nk - \mean{\Nk}}{({\factorm})^{1/2}} \xrightarrow{\cL} \cN(0,\mu^c_k).
\]
\end{enumerate}
\end{thm}
For $\beta_{k,n}$ the theorem above needs two adjustments. Firstly, we need to replace the term $n^{k+1}r_n^{mk}$ with $n^{k+2}r_n^{m(k+1)}$, and $\mu^c_k$ with $\mu_k^b$ (similarly to Theorem \ref{thm:lim_mean_var_subcrit}). Secondly, the proof of the central limit theorem in part 3 is more delicate, and requires an additional assumption that  $nr_n^m \le n^{-\eps}$ for some $\eps > 0$.
 
\subsection{The Critical Range ($n\rnm \to \lambda \in (0,\infty)$)}\label{sec:crit_range}

In the dust phase, $\beta_{0,n}$ was $O(n)$, while the other Betti numbers of $\cU(\cP_n, r_n)$ were of a much lower magnitude. In the critical regime, this behavior changes significantly, and we observe that all the Betti numbers (as well as counts of all critical points) are $O(n)$.
In other words, the behavior of $\cU(\cP_n, r_n)$ is much more complex, in the sense that it consists of many cycles of any dimension $1\le k \le m-1$.

Unfortunately, in the critical regime, the combinatorics of cycle counting becomes highly complicated. However, we can still prove the following qualitative result, which shows that $\mean{\bk} = O(n)$.
%
%

\begin{thm}\label{thm:bounds_betti_crit}
If $n\rnm\to\lambda\in(0,\infty)$, then for $1\le k \le m-1$,
\[
	0 < \liminf_{n\to\infty} n^{-1}{\mean{\bk}} \le \limsup_{n\to\infty} n^{-1}{\mean{\bk}} < \infty.
\]
\end{thm}

Fortunately, the situation with the critical points is much better.
A critical point of index $k$ is always generated by subsets $\cY$ of exactly $k+1$ points. Therefore, nothing essentially changes in our methods when we turn to examine the limits of $\Nk$. We can prove the following limit theorems.
\begin{thm}\label{thm:lim_crit}
If $n\rnm\to \lambda\in(0,\infty)$, then for $1 \le k \le m$,
\begin{align*}
\lim_{\ninf} \frac{\mean{\Nk}}{n}&= {\gamma_{k}(\lambda)}, \\
\lim_{\ninf} \frac{\var{\Nk}}{n}&= {\sigma^2_{k}(\lambda)}, \\
\frac{\Nk - \mean{\Nk}}{\sqrt{n}} &\xrightarrow{\cL} \cN(0,\sigma^2_k(\lambda)).
\end{align*}
where
\[
\gamma_k(\lambda) := \frac{\lambda^k}{(k+1)!}  \int_{\cM}\int_{(\Rm)^{k}}f^{k+1}(x) h^c_1(0,\by) e^{-\lambda \omega_m R^m(0,\by)f(x)}d\by dx,
\]
$R,\ h_\eps^c$, are defined in \eqref{eq:def_R}, \eqref{eq:def_h_eps}, respectively.
 The expression defining $\sigma^2_k(\lambda)$ is rather complicated, and will be discussed in the proof.
\end{thm}

The term  $\omega_m$ stands for the volume of the unit ball in $\R^m$. 
As mentioned above, in general it is difficult to present a more explicit formula for $\gamma_k(\lambda)$. However, for $m\le 3$ and $f\equiv 1$ (the uniform distribution) it is possible to evaluate $\gamma_k(\lambda)$ (using tedious calculus arguments which we omit here). For $m=3$ these computations yield -
\begin{align*}
	\gamma_1(\lambda) &= 4(1-e^{-\frac{4}{3}\pi \lambda}),\\
	\gamma_2(\lambda) &= (1+\frac{\pi^2}{16}) (3- 3e^{-\frac{4}{3}\pi \lambda}-4\pi\lambda e^{-\frac{4}{3}\pi \lambda}),\\
	\gamma_3(\lambda) &=\frac{\pi^2}{48} (9 - 9e^{-\frac{4}{3}\pi \lambda} - 12 \pi \lambda e^{-\frac{4}{3}\pi \lambda} - 8\pi^2 \lambda^2e^{-\frac{4}{3}\pi \lambda}),
\end{align*}
and
\begin{align*}
	\frac{d}{d\lambda}\gamma_1(\lambda) &= \frac{16}{3} \pi e^{-\frac{4}{3}\pi \lambda},\\
	\frac{d}{d\lambda}\gamma_2(\lambda) &= (16+\pi^2)\frac{\pi^2}{3}\lambda e^{-\frac{4}{3}\pi \lambda},\\
	\frac{d}{d\lambda}\gamma_3(\lambda) &= \frac{2}{9}\pi^5 \lambda^2e^{-\frac{4}{3}\pi \lambda},
\end{align*}
where $\frac{d}{d\lambda}\gamma_k(\lambda)$ can be thought of as the rate at which critical points appear.
Figures \ref{fig:gammas}(a) and \ref{fig:gammas}(b)  are the graphs of these curves.

\begin{figure}[h]
\centering
\includegraphics[scale=0.3]{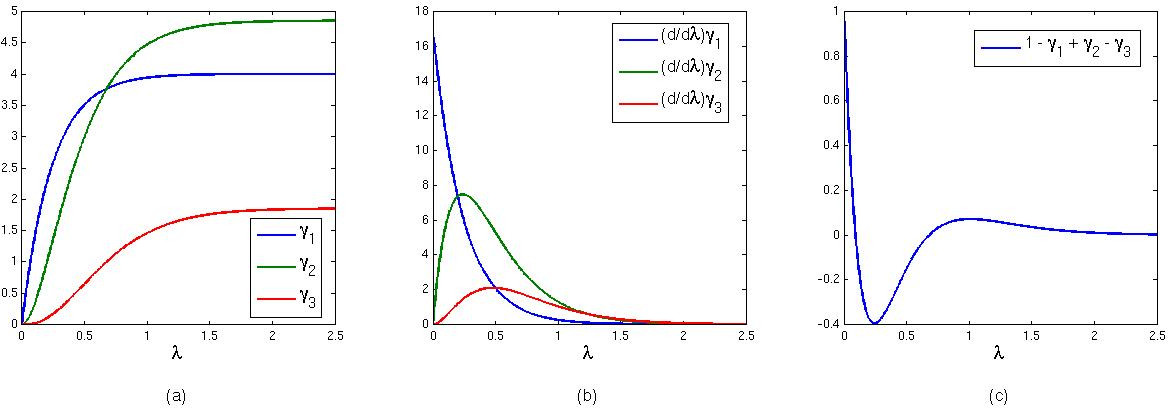}
\caption{\label{fig:gammas} The graphs of the $\gamma_k$ functions for the case where $m=3$, and $f\equiv 1$. (a) The graphs for the limiting number of critical points $\gamma_k(\lambda)$. (b) The graphs for the rate of appearance of critical points given by $\frac{d}{d\lambda}\gamma_k(\lambda)$. (c) The limiting (normalized) Euler characteristic given by $1 - \gamma_1(\lambda) + \gamma_2(\lambda) - \gamma_3(\lambda)$. }
\end{figure}

As mentioned earlier, in this regime we cannot get exact limits for the Betti numbers. However, we can use the limits of the critical points to compute the limit of another important topological invariant of $\cU(\cP_n,r_n)$ -- its Euler characteristic. The Euler characteristic $\chi_n$ of $\cU(\cP_n, r_n)$ (or, equivalently, of $\CC(\cP_n, r_n)$) has a number
of equivalent definitions. One of the definitions, via Betti numbers, is
\begin{equation}\label{eq:ec_betti}
    \chi_n = \sum_{k=0}^{m} (-1)^k \beta_{k,n}.
\end{equation}
In other words, the Euler characteristic ``summarizes" the information contained in Betti numbers to a single integer.
Using Morse theory, we can also compute $\chi_n$  from the critical points of the distance function by 
\[
	\chi_n = \sum_{k=0}^{m} (-1)^kN_{k,n}.
\]
Thus, using Theorem \ref{thm:lim_crit} we have the following result.
\begin{cor}
If $n\rnm \to \lambda \in (0,\infty)$, then
\[
\limninf n^{-1} \mean{\chi_n} = 1+\sum_{k=1}^m {(-1)^k \gamma_k(\lambda) }.
\]
\end{cor}
This limit provides us with partial, yet important, topological information about the complex $\cU(\cP_n,r_n)$ in the critical regime. While we are not able to derive the precise limits for each of the Betti numbers individually, we can provide the asymptotic result for their ``summary". In addition, numerical experiments (cf. \cite{kahle_limit_2010}) seem to suggest that at different ranges of radii there is at most a single degree of homology which dominates the others. This implies that $\chi_n \approx (-1)^k\beta_{k,n} $ for the appropriate range. If this heuristic could be proved in the future, the result above could be used to approximate $\beta_{k,n}$ in the critical regime.
In Figure \ref{fig:gammas}(c) we present the curve of the limit Euler characteristic (normalized) for $m=3$ and $f\equiv 1$.
Finally, we note that while we presented the limit  for the first moment of the Euler characteristic, using Theorem \ref{thm:lim_crit} one should be able to prove stronger limit results as well.

\subsection{The Super-Critical Range ($n\rnm\to \infty$)}\label{sec:results_supcrit}

Once we move from the critical range into the super-critical, the complex $\cU(\cP_n,r_n)$ becomes more and more connected, and less porous. The ``noisy" behavior (in the sense that there are many holes of any possible dimension) we observed in the critical regime vanishes. This, however does not happen immediately.
The scale at which major changes occur is when $n\rnm \propto \log n$.

The main difference between this regime and the previous two, is that while the number of critical points is still $O(n)$, the Betti numbers are of a much lower magnitude. In fact, for $r_n$ big enough, we observe that $\bk \sim \beta_k(\cM)$, which implies that these values are $O(1)$.

For the super-critical phase we have to assume that $\fmin := \inf_{x\in\cM} f(x) >0$. This condition is required for the proofs, but is not a technical issue only. Having a point $x\in \cM$ where $f(x)=0$  implies that in the vicinity of $x$ we expect to have relatively few points in $\cP_n$. Since the radius of the balls generating $\cU(\cP_n, r_n)$ goes to zero, this area might become highly porous or disconnected , and look more similar to other regimes. However, we postpone this study for future work.

We start by describing the limit behavior of the critical points, which is very similar to that of the critical regime.
\begin{thm}\label{thm:lim_supcrit}
If $r_n\to 0$, and $n\rnm \to \infty$, then for $1 \le k \le m$,
\begin{align*}
\lim_{\ninf} \frac{\mean{\Nk}}{n}&= {\gamma_{k}(\infty)}, \\
\lim_{\ninf} \frac{\var{\Nk}}{n}&= {\sigma^2_{k}(\infty)},  \\
\frac{\Nk - \mean{\Nk}}{\sqrt{n}} &\xrightarrow{\cL} \cN(0,\sigma^2_k(\infty)).
\end{align*}
where
\[
\gamma_k(\infty) = \lim_{\lambda\to\infty} \gamma_k(\lambda) = \frac{1}{(k+1)!}\int_{(\Rm)^k} {h^c(0,\by) e^{-\omega_m R^m(0,\by)} \,d\by},
\]
$R,\ h^c$, are defined in \eqref{eq:def_R}, \eqref{eq:def_h}, respectively.
\end{thm}

$\omega_m$ is the volume of the unit ball in $\R^m$
The combinatorial analysis of the Betti numbers $\bk$ in the super-critical regime suffers from the same difficulties described in the critical regime. However, in the special case that $r_n$ is big enough so that $\cU(\cP_n,r_n)$ covers $\cM$, we can use a different set of methods to derive limit results for $\bk$.

\subsubsection*{The Coverage Regime}
In \cite{penrose_random_2003}(Section 13.2), it is shown that for samples generated on a $m$-dimensional torus, the complex $\cU(\cP_n, r_n)$ becomes connected when $nr_n^m \approx (\omega_m \fmin 2^m)^{-1} \log n$. This result could be easily extended to the general class of manifolds studied in this paper (although we will not pursue that here). 
While the complex is reaching a finite number of components ($\beta_{0,n} \to \beta_0(\cM)$), it is still possible for it to have very large Betti numbers for $k\ge 1$. 
In this paper we are interested in a threshold for which we have $\bk = \beta_k(\cM)$ for all $k$ (and not just $\beta_0$).
We will show that this threshold is when $nr_n^m = (\omega_m \fmin)^{-1}\log n$, so that $r_n$ is twice than the radius required for connectivity.

To prove this result we need two ingredients. The first one is a coverage statement, presented in the following proposition.

\begin{prop}[Coverage]\label{prop:lim_coverage}
	If $n\rnm \ge C \log n$, then:
	\begin{enumerate}
		\item If $C> (\omega_m \fmin)^{-1}$, then
		\[
			\limninf \prob{\cM \subset \cU(\cP_n, r_n)} = 1.
		\]
		\item If $C >2(\omega_m \fmin)^{-1}$, then almost surely there exists $M>0$ (possibly random), such that for every $n>M$ we have $\cM \subset \cU(\cP_n, r_n)$.			
	\end{enumerate}
\end{prop}

The second ingredient is a statement about the critical points of the distance function, unique to the coverage regime.
Let $\hr_n$ be any sequence of positive numbers such that (a) $\hr_n \to 0$, and (b) $\hr_n > r_n$ for every $n$. Define $\Nkt$ to be the number of critical points of $d_{\cP_n}$ with critical value bounded by $\hr_n$.
Obviously, $\Nkt \ge \Nk$, but we will prove that choosing $r_n$ properly, these two quantities are asymptotically equal.

\begin{prop}\label{prop:lim_glob}
	If $n\rnm \ge C \log n$, then:
	\begin{enumerate}
	\item
	If $C > (\omega_m \fmin)^{-1}$, then
	\[
		\limninf \prob{\Nk = \Nkt, \ \forall 1\le k \le m} = 1.
	\]
	\item If $C>2(\omega_m \fmin)^{-1}$, then almost surely there exists $M>0$ (possibly random), such that for $n>M$
	\[
		\Nk = \Nkt,\quad \forall 1\le k \le m.
	\]
	\end{enumerate}
\end{prop}
In other words, if $r_n$ is chosen properly, then $\cU(\cP_n, r_n)$ contains all the `local' (small valued) critical points of $d_{\cP_n}$.

Combining the fact that $\cM$ is covered, the deformation retract argument in \cite{niyogi_finding_2008}, and the fact that there are no local critical points outside $\cU(\cP_n, r_n)$, using Morse theory, we have the desired statement about the Betti numbers.

\begin{thm}[Convergence of the Betti Numbers]\label{thm:lim_betti_sup}
	If $r_n \to 0$, and $n\rnm \ge C \log n$, then:
	\begin{enumerate}
	\item
	If $C > (\omega_m \fmin)^{-1}$, then
	\[
		\limninf \prob{\bk = \beta_k(\cM),\ \forall 0\le k \le m} = 1.
	\]
	\item If $C>2(\omega_m \fmin)^{-1}$, then almost surely there exists $M>0$, such that for $n>M$
	\[
		\bk = \beta_k(\cM),\quad \forall 0\le k \le m.
	\]
	Note that $M$ (the exact point of convergence) is  random.
	\end{enumerate}
\end{thm}

A common problem in topological manifold learning is the following:

\noindent \emph{Given a set of random points $\cP$, sampled  from an unknown manifold $\cM$, how can one infer the topological features of $\cM$?}

\noindent
The last theorem provides a possible solution. Draw balls around $\cP$, with a radius $r$ satisfying the condition in Theorem \ref{thm:lim_betti_sup}. As the sample size grows it is guaranteed that the Betti numbers computed from the union of the balls will recover those of the original manifold $\cM$. This solution is in the spirit of the result in \cite{niyogi_finding_2008}, where a bound on the recovery probability is given as a function of the sample size and the condition number of the manifold, for a uniform measure on $\cM$. The result in \ref{thm:lim_betti_sup} applies for a larger class of probability measures on $\cM$,  require much weaker assumptions on the geometry of the manifold (the result in \cite{niyogi_finding_2008} requires the knowledge of the condition number, or the reach, of the manifold), and convergence is shown to occur almost surely.

\section{Proofs}
\label{proofs}
In this section we provide proofs for the statements in this paper. 
We note that the proofs of theorems \ref{thm:lim_mean_var_subcrit} - \ref{thm:lim_supcrit} are similar to the proofs of the equivalent statements in \cite{kahle_random_2011, kahle_limit_2010} (for the Betti numbers), and in \cite{bobrowski_distance_2011} (for the critical points). There are, however, significant differences when dealing with samples on a closed manifold. We provide detailed proofs for the limits of the first moments, demonstrating these differences, and refer the reader to \cite{bobrowski_distance_2011, kahle_random_2011, kahle_limit_2010} for the rest of the details.

\subsection{Some Notation and Elementary Considerations}
This section is devoted to prove  the results
in Section \ref{sec:results}, and is organized
according to
situations: sub-critical (dust), critical, and super-critical. In this section
we list some common notation and note some simple facts that will be used
in the proofs.

\begin{list}{\labelitemi}{\leftmargin=1em}
\item Henceforth, $k$ will be fixed, and whenever we use $\cY,\cY'$ or $\cY_i$  we implicitly assume (unless stated otherwise) that either $\abs{\cY}=\abs{\cY'} = \abs{\cY_i} = k+2$ for $k$-cycles, or $\abs{\cY}=\abs{\cY'} = \abs{\cY_i} = k+1$ for index $k$ critical points.
\item Usually, finite subsets of $\R^d$ will be denoted calligraphically ($\cX,\cY$). However inside integrals we  use boldfacing and lower case ($\bx,\by$).
\item For every $x\in\cM$ we denote by $T_x\cM$ the tangent space of $\cM$ at $x$, and define $\exp_x:T_x\cM\to \cM$ to be the exponential map at $x$. Briefly, this means that for every $v\in T_x\cM$, the point $\exp_x(v)$ is the point on the unique geodesic leaving $x$ in the direction of $v$, after traveling a geodesic distance equal to $\norm{v}$.
\item For $x\in\R^d$, $\bx \in \cM^{k+1}$ and $\by\in(\R^m)^{k}$, we use the shorthand
\begin{align*}
f(\bx) &:= f(x_1)f(x_2)\cdots f(x_{k+1}), \\
f(x,\exp_x(\bv)) &:= f(x) f(\exp_x(v_1))\cdots f(\exp_x(v_k)),\\
h(0,\by) &:= h(0,y_1,\ldots,y_k).
\end{align*}
\end{list}

Throughout the proofs we will use the following notation. Let $x\in \cM$, and let $v\in T_x\cM$ be a tangent vector. We define
\[
	\nabla_{\eps}(x,v) = \frac{\exp_x(\eps v) - x}{\eps}.
\]
By definition, it follows that
\[
	\lim_{\eps\to 0} \nabla_{\eps}(x,v) = v.
\]
The following lemmas will be useful when we will be required to approximate geodesic distances and volumes by Euclidean ones.

\begin{lem}\label{lem:geodesic_bounds}
Let $\delta > 0$.
If $\norm{\nabla_\eps(x,v)} \le C$ for all $\eps>0$, and for some $C>0$. Then there exists a small enough $\tilde{\eps} > 0$ such that for every $\eps < \tilde{\eps}$
\[
	\norm{v} \le C(1+\delta).
\]
\end{lem}
\begin{proof}
If $\norm{\nabla_\eps(x,v)} \le C$, then the $(C\eps)$-tube around $\cM$, contains the line segment connecting $x$ and $\exp_x(\eps v)$. Therefore, using Theorem 5 in \cite{memoli2005distance} we have that
\[
	\frac{\norm{\eps v}}{\norm{x-\exp_x(\eps v)}} \le 1 + C'\sqrt{\eps}.
\]
This implies that
\[
	\norm{v} \le (1+C'\sqrt{\eps}) \norm{\nabla_\eps(x,v)},
\]
for some $C'>0$.
Therefore, if $\eps$ is small enough we have that
\[
	\norm{v} \le C(1+\delta),
\]
which completes the proof.
\end{proof}

Throughout the proofs we will repeatedly use two different occupancy probabilities, defined as follows,
\begin{align}
	p_b(\cY, \eps) &:= \int_{\cU(\cY,\eps)\cap \cM} f(\xi)d\xi  \label{eq:def_pb}\\
	p_c(\cY) &:= \int_{B(\cY)\cap \cM} f(\xi)d\xi, \label{eq:def_pc}
\end{align}
where $B(\cY)$ is defined in \eqref{eq:def_B}.
The next lemma is a version of Lebesgue differentiation theorem, which we will be using.
\begin{lem}\label{lem:lim_prob_vol}
For every $x\in\cM$ and $\by \in (T_x(\cM))^k$, if $r_n\to 0 $, then
\begin{enumerate}
\item
\[
	\limninf \frac{p_b((x,\exp_x(r_n\by)), r_n)}{r_n^m V(0,\by)} = f(x),
\]
where $V(\cY) = \vol(\cU(\cY,1))$.
\item	\[
		\limninf \frac{p_c(x,\exp_x(r_n\by))}{\rnm\omega_m R^m(0,\by)} = f(x),
	\]
	where $\omega_m$ is the volume of a unit ball in $\R^m$.
\end{enumerate}

\end{lem}
\begin{proof}
We start with the proof for $p_c$. Set $B_n := B(x,\exp_x(r_n \by))\subset \R^d$. Then
\[
p_c(x,\exp_x(r_n \by)) =  \int\limits_{B_n \cap \cM} f(\xi)d\xi.
\]
Next, use the change of variables $\xi \to \exp_x(r_n v)$, for $v\in T_x\cM\simeq \R^m$.
Then,
\begin{equation}\label{eq:p_int}
p_c(x,\exp_x(r_n \by)) = r_n^m \int_{\Rm} f(\exp_x(r_n v)) \ind\set{\exp_x(r_n v) \in B_n} J_x(r_n v) dv,
\end{equation}
where $J_x(v) = \frac{\partial \exp_x}{\partial v}$.

We would like to apply the Dominated Convergence Theorem (DCT) to this integral, to find its limit. First, assuming that the DCT condition holds, we find the limit.
\begin{itemize}
\item By definition, $\exp_x(r_n v) \to x$, and therefore,
\[
	\limninf f(\exp_x(r_nv)) = f(x).
\]
\item Note that the function $H(v,\by) := \ind\set{v\in B(0,\by))}$ is almost everywhere continuous in $\R^d\times(\R^d)^k$, and also that
\[
	\ind\set{\exp_x(r_n v)\in B_n} = H(\nabla{r_n}(x,v), \nabla{r_n}(x,\by)).
\]
Since $\nabla{r_n}(x,v) \to v$, and $\nabla{r_n}(x,\by) \to \by$ (when $n\to\infty$), we have that for almost every $v,\by$,
\[
	\limninf\ind\set{\exp_x(r_n v)\in B_n} = H(v, \by) = \ind\set{v\in B(0,\by)}.
\]
\item By definition,
\[
	\limninf J_x(r_nv) = 1.
\]
\end{itemize}
Putting it all together, we have that
\[
	\limninf r_n^{-m}p_c(x,\exp_x(r_n\by)) = f(x) \vol_m(B(0,\by)) = f(x) \omega_m R^m(0,\by),
\]
which is the limit we are seeking.

To conclude the proof we have to show that the DCT condition holds for the integrand in \eqref{eq:p_int}. For a fixed $\by$, for every $v$ for which the integrand is nonzero, we have that $\exp_x(r_nv) \in B_n$ which implies that
\[
	\norm{\nabla_{r_n}(x,v)} \le 2R(0,\nabla_{r_n}(x,\by)).
\]
Since $R(0,\nabrn(x,\by)) \to R(0,\by)$, we have that $n$ for large enough
\[
	\norm{\nabrn(x,v)} \le 3R(0,\by),
\]
Using Lemma \ref{lem:geodesic_bounds} we then have that
\[
	\norm{v} \le 3(1+\delta) R(0,\by),
\]
for some $\delta>0$.
This means that the support of the integrand in \eqref{eq:p_int} is bounded. Since $f$ is bounded, and $J_x$ is continuous, we deduce that the integrand is well bounded, and we can safely apply the DCT to it.

The proof for $p_b$ follows the same line of arguments, replacing $B_n$ with   $$U_n := \cU((x,\exp_x(r_n\by)),r_n).$$ To bound the integrand we use the fact that if $\exp_x(r_n v) \in U_n$, then
\[
	\norm{\nabrn(x,v)} \le \diam(\cU(0, \nabrn(x,\by), 1)),
\]
and as $n\to \infty$, we have $\diam(\cU(0, \nabrn(x,\by), 1))\to \diam(\cU(0,\by),1)$.
\end{proof}

In \cite{bobrowski_distance_2011,kahle_random_2011,kahle_limit_2010} full proofs are presented for statements similar to those in this paper, only for sampling in Euclidean spaces rather than compact manifolds.
The general method of proving statements on compact manifold is quite similar, but important adjustments are required. We are going to present those adjustments for proving the basic claims, and refer the reader to the proofs in \cite{bobrowski_distance_2011,kahle_random_2011,kahle_limit_2010} taking into consideration the necessary adjustments.

\subsection{The Sub-Critical Range ($n\rnm \to 0$)}

\begin{proof}[Proof of Theorem \ref{thm:lim_mean_var_subcrit}]
We give a full proof for the limit expectations for both the Betti numbers and critical points, and then discuss the limit of the variances.

\noindent\textbf{The expected number of critical points:}\\
From the definition of $\Nk$ (see \eqref{eq:def_nk}), using the fact that index-$k$ critical points are generated by subsets of size $k+1$ (see Definition \ref{def:crit_pts}),  we can compute $\Nk$ by iterating over all possible subsets of $\cP_n$ of size $k+1$ in the following way,	
\[
	\Nk = \sum_{\cY \subset \cP_n} \grn^c(\cY, \cP_n),
\]
where $g_\eps$ is defined in \eqref{eq:def_g_eps}. Using Palm theory (Theorem \ref{thm:palm}), we have that
\begin{equation}\label{eq:mean_nk_step1}
	\mean{\Nk} = \frac{n^{k+1}}{(k+1)!} \mean{\grn^c(\cY', \cY' \cup \cP_n)},
\end{equation}
where $\cY'$ is a set of $\iid$ random variables, with density $f$, independent of $\cP_n$. Using the definition of $\grn$, we have that
\[
\mean{\grn^c(\cY', \cY'\cup \cP_n)} = \mean{\cmean{\grn^c(\cY', \cY'\cup \cP_n)}{\cY'}} = \mean{\hrn^c(\cY')e^{-np_c(\cY')}},
\]
where $p_c$ is defined in \eqref{eq:def_pc}.
Thus,
\[
	\mean{\grn^c(\cY', \cY'\cup\cP_n)} = \int_{\cM^{k+1}} f(\bx) \hrn^c(\bx) e^{-np_c(\bx)} d\bx.
\]
To evaluate this integral, recall that $\bx = (x_0, \ldots, x_k) \in \cM^{k+1}$ and  use the following change of variables
\[
    x_0 \to x \in \cM, \qquad x_i \to \exp_x(v_i),\ v_i \in T_x\cM \simeq \R^m,
\]
then,
\begin{align*}
&\mean{\grn^c(\cY', \cY'\cup \cP_n)} \\
&\quad= \int_{\cM} \int_{(T_x\cM)^k}f(x,\exp_x(\bv)) h_{r_n}^c(x,\exp_x(\bv))  e^{-np_c(x,\exp_x(\bv))} J_x(\bv) d\bv dx.
\end{align*}
where $\bv = (v_1,\ldots, v_k)$,  $\exp_x(\bv) = (\exp_x(v_1),\ldots, \exp_x(v_k))$, and $J_x(v) = \frac{\partial{\exp_x}}{\partial{v}}$.
From now on we will think of $v_i$ as vectors in $\R^m$. Thus, the change of variables $v_i \to r_n y_i$ yields,
\begin{equation}\label{eq:mean_nk_step2}
\begin{split}
&\mean{\grn^c(\cY', \cY'\cup \cP_n)} \\
&\quad = r_n^{mk}\int_{\cM} \int_{(\Rm)^k}f(x,\exp_x(r_n\by)) h_{r_n}^c(x,\exp_x(r_n\by)) e^{-np_c(x,\exp_x(r_n \by))} J_x(r_n\by) d\by dx.
\end{split}
\end{equation}
The integrand above admits the DCT conditions, and therefore we can take a point-wise limit. We compute the limit now, and postpone showing that the integrand is bounded to the end of the proof.

Taking the limit term by term, we have that:
\begin{itemize}
\item $f$ is continuous almost everywhere in $\cM$, therefore
\[
\limninf f(\exp_x(r_n y_i)) = f(x)
\]
for almost every $x\in \cM$.

\item The discontinuities of the function $h_1^c:(\Rd)^{k+1} \to \set{0,1}$ are either subsets $\bx$ for which $C(\bx)$ is on the boundary of $\conv(\bx)$, or where $R(\bx) = 1$. This entire set has a Lebesgue measure zero in $(\R^d)^{k+1}$. Therefore, we have
\[
\limninf \hrn^c(x, \exp_x(r_n \by)) =  \limninf h_1^c(0, \nabla_{r_n}(x, \by)) = h_1(0,\by),
\]
for almost every $x,\by$.

\item Using Lemma \ref{lem:lim_prob_vol}, and the fact that $n\rnm\to 0$, we have that
\[
	\limninf e^{-np_c(x,\exp_x(r_n \by))} = 1.
\]

\item Finally, $\limninf J_x(r_n y_i) = J_x(0) = 1$.

\end{itemize}

Putting all the pieces together (rolling back to \eqref{eq:mean_nk_step1} and \eqref{eq:mean_nk_step2}), we have that
\[
	\limninf (\factorm)^{-1} \mean{\Nk} = \mu_k^c.
\]

Finally, to justify the use of the DCT, we need to find an integrable bound for the integrand in \eqref{eq:mean_nk_step2}.

The main step would be to show that the integration over $(y_1,\ldots, y_k)$ is done over a bounded region in $(\Rm)^k$.
First, note that if $\hrn^c(x,\exp_x(r_n\by)) = h_1^c(0,\nabrn(x,\by)) = 1$, then necessarily  $R(0, \nabrn(x,\by)) < 1$. This implies that $\norm{\nabrn(x, y_i)} <  2$.
Using Lemma \ref{lem:geodesic_bounds}, and the fact that $r_n\to 0$, we can choose $n$ large enough so that $\norm{y_i} < 3$ for every $i$. In other words, we can assume that the integration $dy_i$ is over $B_3(0)\subset \R^m$ only.

Next, we will bound each of the terms in the integrand in \eqref{eq:mean_nk_step2}.
\begin{itemize}
\item The density function $f$ is bounded, therefore,
\[
	f(x,\exp_x(r_n\by)) = f(x)f(\exp_x(r_n \by)) \le f(x) \fmax^k,
\]
where $\fmax := \sup_{x\in \cM} f(x)$.
\item The term $\hrn^c(x,\exp_x(r_n\by)) e^{-np_c(x,\exp_x(r_n\by))}$ is bounded from above by $1$.

\item The function $J_x(v)$ is continuous in $x,v$. Therefore, it is bounded in the compact subspace $\cM\times B_3(0)$, by some constant $C$. Since we know that $y_i\in B_3(0)$, then for $n$ large enough (such that $r_n<1$ we have that $J_x(r_n\by) \le C^k$.
\end{itemize}

Putting it all together, we have that the integrand in \eqref{eq:mean_nk_step2} is bounded by $f(x)\times \mathrm{const}$, and since we proved that the $y_i$-s are bounded, we are done.
\\

\noindent\textbf{The expected Betti numbers:}\\
As mentioned in Section \ref{sec:cech}, most of the results for $\bk$ will be proved using the \cech complex $\CC(\cP_n,r_n)$ rather than the union $\cU(\cP_n,r_n)$. From the Nerve theorem, the Betti numbers of these spaces are equal.

 The smallest simplicial complex forming a non-trivial $k$-cycle is the boundary of a $(k+1)$-simplex which consists of $k+2$ vertices.
Recall that for $\cY\in (\Rd)^{k+2}$, $h_\eps^b(\cY)$ is an indicator function testing whether $\CC(\cY, \eps)$  forms a non-trivial $k$-cycle (see \eqref{eq:def_hb}), and define
\[
 g_\eps^b(\cY,\cP) := h_\eps^b(\cY)\ind\set{\CC(\cY,\eps) \textrm{ is a connected component of } \CC(\cP,\eps)}.
\]
Then iterating over all possible subsets $\cY$ of size $k+2$ we have that
\begin{equation}\label{eq:def_sk}
\Sk := \sum_{\cY \subset \cP_n} \grn^b (\cY, \cP_n),
\end{equation}
is the number of minimal isolated cycles in $\CC(\cP_n,r_n)$. Next, define $\Fk$ to be the number of $k$ dimensional faces in $\CC(\cP_n, r_n)$ that belong to a component with at least $k+3$ vertices. Then
\begin{equation}\label{eq:ineq_bk}
	\Sk \le \bk \le \Sk + \Fk.
\end{equation}
This stems from three main facts:
\begin{enumerate}
\item Every cycle which is not accounted for by $\Sk$ belongs to a components with at least $k+3$ vertices.

\item If $C_1, C_2,\cdots, C_m$ are the different connected components of a space $X$, then
\[
\beta_k(X) = \sum_{i=1}^m\beta_k(C_i).
\]
\item For every simplicial complex $C$ it is true that $\beta_k(C) \le F_k(C)$, where $F_k$ is the number of $k$-dimensional simplices.
\end{enumerate}
For more details regarding the inequality in \eqref{eq:ineq_bk}, see the proof of the analogous theorem in \cite{kahle_limit_2010}.

Next, we should find the limits of $\Sk$  and $\Fk$. For $\Sk$, from \eqref{eq:def_sk} using Palm theory (Theorem \ref{thm:palm}) we have that
\[
	\mean{\Sk} = \frac{n^{k+2}}{(k+2)!} \mean{\grn^b(\cY', \cY'\cup \cP_n)},
\]
where $\cY'$ is a set of $k+2$ $\iid$ random variables with density $f$, independent of $\cP_n$. Using the definition of $\grn^b$ we have that
\[
\mean{\grn^b(\cY', \cY'\cup\cP_n)} = \mean{\cmean{\grn^b(\cY',\cY'\cup\cP_n)}{\cY}} = \mean{\hrn^b(\cY')e^{-np_b(\cY', 2r_n)}},
\]
where $p_b$ is defined in \eqref{eq:def_pb}.
Following the same steps as in the proof for the number of critical points, leads to
\[
	\limninf (\bfactorm)^{-1} \mean{\Sk} = \mu_k^b.
\]
Thus, to complete the proof we need to show that
$(\bfactorm)^{-1} \mean{\Fk}  \to 0$. To do that, we consider sets $\cY$ of $k+3$ vertices, and define
\[
h_\eps^f (\cY) := \ind\set{\CC(\cY, \eps) \textrm{ is connected and contains a $k$-simplex}}.
\]
Then,
\[
	\Fk \le \binom{k+3}{k+1} \sum_{\cY\subset \cP_n}\hrn^f(\cY).
\]
Using Palm Theory, we have that
\[
	\mean{\Fk} \le \frac{n^{k+3}}{2(k+1)!} \mean{\hrn^f(\cY)}.
\]
Since $\hrn^f$ requires that $\CC(\cY,r_n)$ is connected, similar localizing arguments to the ones used previously in this proof show that
\[
	\limninf (n^{k+3} r_n^{m(k+2)})^{-1}\mean{\Fk} < \infty.
\]
Thus, since $n\rnm \to 0$, we have that
\[
	\limninf (n^{k+2} r_n^{m(k+1)})^{-1}\mean{\Fk} =0,
\]
which completes the proof.

For $\beta_{0,n}$, using Morse theory we have that $N_{0,n}-N_{1,n} \le \beta_{0,n} \le N_{0,n}$. Since $\mean{N_{0,n}} = n$, and $n^{-1}\mean{N_{1,n}} \to 0$, we have that
$\limninf n^{-1}\mean{\beta_{0,n}} = 1$.
\\

\noindent\textbf{The limit variance:}\\
To prove the limit variance result, the computations are similar to the ones in \cite{bobrowski_distance_2011,kahle_limit_2010}. The only adjustment required is to change the domain of integration to be $\cM$ instead of $\Rd$, the same way we did in proving the limit expectations. We refer the reader to Appendix \ref{sec:apndx_moments}  for an outline of these proofs.
\end{proof}

\begin{proof}[Proof of Theorem \ref{thm:dist_subcrit}]
We start with the case when $\factorm \to 0$.
In this case, the $L^2$ convergence is a direct result of the fact that
\[
	\limninf\mean{\Nk} = \limninf\var{\Nk} =0.
\]
Next, observe that
\[
	\prob{\Nk > 0} \le \mean{\Nk},
\]
and since $(\factorm)^{-1}\mean{\Nk}\to 0$, there exists a constant $C$ such that
\[
	\prob{\Nk >0} \le C \factorm.
\]
Thus, if $\sum_{n=1}^\infty \factorm < \infty$, we can use the Borel-Cantelli Lemma, to conclude that a.s.\! there exists $M>0$ such that for every $n>M$ we have $\Nk = 0$. This completes the proof for the first case.

For the other cases, we refer the reader to \cite{bobrowski_distance_2011,kahle_limit_2010}. The proofs in these papers use Stein's method (see Appendix \ref{sec:apndx_stein}), and mostly rely on moments evaluation (up to the forth moment). We observed in the previous proof that moment computation in the manifold case is essentially the same as in the Euclidean case, and therefore all that is needed are a few minor adjustments.
\end{proof}

\subsection{The Critical Range ($n\rnm\to \lambda$)}
We prove the result for the number of critical points first.

\begin{proof}[Proof of Theorem \ref{thm:lim_crit}]
For the critical phase, we start the same way as in the proof of Theorem \ref{thm:lim_supcrit}. All the steps and bounds are exactly the same, the only difference is in the limit of the exponential term inside the integral in \eqref{eq:mean_nk_step2}. Using Lemma \eqref{lem:lim_prob_vol}, and the fact that $n\rnm \to \lambda$ we conclude that,
\[
\limninf e^{-np_c(x,\exp_x(r_n\by))} = e^{-\lambda \omega_m R^m(0,\by)f(x)}.
\]
Thus, we have
\begin{align*}
&\limninf(\factorm)^{-1} \mean{\Nk} \\
&\quad=
\frac{1}{(k+1)!}  \int_{\cM}\int_{(\Rm)^{k}}f^{k+1}(x) h^c_1(0,\by) e^{-\lambda \omega_m R^m(0,\by)f(x)}d\by dx,
\end{align*}
and using the fact that $\factorm \sim n\lambda^k$ completes the proof.

For the proofs for the variance and the CLT we refer the reader to Appendix \ref{sec:apndx_moments} and \cite{bobrowski_distance_2011}.
\end{proof}

\begin{proof}[Proof of Theorem \ref{thm:bounds_betti_crit}]
From the proof of Theorem  \ref{thm:lim_mean_var_subcrit} we know that
\[
	\Sk \le \bk \le \Sk + \Fk.
\]
Similar methods to the ones we used above, can be used to show that
\begin{align*}
	&\limninf{(\bfactorm)^{-1}} \mean{\Sk} \\
	&\quad= \frac{1}{(k+2)!} \int_{\cM}\int_{(\R^m)^{k+1}} f^{k+2}(x) h^b_1(0,\by) e^{-\lambda 2^m V(0,\by)f(x)}d\by dx,
\end{align*}
where $V(\cY) = \vol(\cU(\cY,1))$ (see Lemma \ref{lem:lim_prob_vol}), and also that
\[
	\limninf{(n^{k+3} r_n^{m(k+2)})^{-1}} \mean{\Fk} < \infty.
\]
Since $n\rnm \to \lambda$, we have that $\bfactorm \sim n \lambda^{k+1}$. Thus we have shown that
\[
	A n \le \mean{\bk} \le B n,
\]
for some positive constants $A,B$, which completes the proof.
\end{proof}

\subsection{The Super-Critical Range ($n\rnm\to \infty$)}
\begin{proof}[Proof of Theorem \ref{thm:lim_supcrit}]
For the super-critical regime, we repeat the steps we took in the other phases, with the main difference being that instead of using the change of variables $x_i \to \exp_x(r_ny_i)$, we now use $x_i \to \exp_x(s_n y_i)$ where $s_n = n^{-1/m}$.
Thus, instead of the formula in \eqref{eq:mean_nk_step2} we now have
\begin{equation}\label{eq:mean_nk_crit}
\begin{split}
&\mean{\hrn(\cY)e^{-n p_c(\cY)}} =\\
&\qquad n^{-k}\int_{\cM} \int_{(\Rm)^k}f(x,\exp_x(s_n\by)) h^c_{r_n}(x,\exp_x(s_n\by)) e^{-np_c(x,\exp_x(s_n \by))} J_x(s_n\by) d\by dx.
\end{split}
\end{equation}
As we did before, we wish to apply the DCT to the integral in \eqref{eq:mean_nk_crit}. We will compute the limit first, and show that the integrand is bounded at the end.
\begin{itemize}
\item As before we have
\[
\limninf f(x,\exp_x(s_n \by)) = f^{k+1}(x).
\]
\item The limit of the indicator function is now a bit different.
\begin{align*}
	\hrn^c(x,\exp_x(s_n \by)) &= h^c_1 (0, r_n^{-1} s_n  \nabsn(x,\by))
	\\
	&= h^c(0, \nabsn(x,\by)) \ind\set{r_n^{-1} s_n R(0,\nabsn(x,\by)) < 1} .
\end{align*}
Now, since $R(0,\nabsn(x,\by)) \to R(0,\by)$ and $r_n^{-1}s_n \to 0$, we have that
\[
	\limninf \hrn^c(x,\exp_x(s_n \by)) = h^c(0, \by).
\]

\item Using Lemma \ref{lem:lim_prob_vol} we have that
\[
\limninf \frac{p_c(x,\exp_x(s_n\by))}{s_n^m \omega_m R^m(0,\by)} = f(x).
\]
This implies that
\[
	\limninf e^{-np_c(x,\exp_x(s_n \by))} = e^{-\omega_m R^m(0, \by) f(x)}.
\]
\end{itemize}
These computations yield,
\[
	\limninf n^{-1} \mean{\Nk} = \frac{1}{(k+1)!} \int_{\cM}\int_{(\Rm)^k} f^{k+1}(x) h^c(0,\by) e^{-\omega_m R^m(0, \by) f(x)} d\by dx.
\]

Finally, for the inner integral, use the following change of variables - $y_i \to (f(x))^{-1/m} v_i$, so that $d\by = f^{-k}(x)d\bv$. This yields,
\[
	\limninf n^{-1}\mean{\Nk} = \frac{1}{(k+1)!} \int_{\cM}\int_{(\Rm)^k} f(x) h^c(0,\bv) e^{-\omega_m R^m(0,\bv)} d\bv dx.
\]
Using the fact that $\int_{\cM}f(x)dx = 1$ completes the proof.

It remains to show that the DCT condition applies to the integral in \eqref{eq:mean_nk_crit}. The main difficulty in this case stems from the fact that the variables $y_i$ are no longer bounded. Nevertheless, we can still bound the integrand, taking advantage of the exponential term.
\begin{itemize}
\item As before, we have $f(x,\exp_x(s_n,\by)) \le f(x) \fmax^k$.

\item Being an indicator function, it is obvious that $\hrn^c(x,\exp_x(s_n\by)) \le 1$.
\item To bound the exponential term from above, we will find a lower bound to $p_c(x,\exp_x(s_n\by))$.
Define a function $G:\cM\times (\R^m)^k \times [0,1]\to \R$ as follows,
\[
	G(x, \bv, \rho) = \begin{cases}
		\frac{p_c(x, \exp_x(\rho \bv))}{\omega_m R^m(0, \rho \bv) f(x)}& \rho > 0, \\
		1 & \rho = 0.
	\end{cases}
\]
From Lemma \ref{lem:lim_prob_vol} we know that $G$ is continuous in the compact subspace $\cM \times (B_3(0))^k \times [0,1]$, and thus uniformly continuous. Therefore, for every $\alpha > 0, x\in \cM,\bv\in (B_3(0))^k$, there exists $\tilde{\rho} > 0 $ such that for every $\rho < \tilde{\rho}$ we have
\[
	G(x, \bv, \rho) \ge 1-\alpha.
\]
Now, consider $\bv = \frac{s_n}{r_n} \by$, then as we proved in the sub-critical phase, $\bv \in (B_3(0))^k$. Thus, for $n$ large enough (such that $r_n < \tilde{\rho}$), we have that for every $x,\by$
\[
 \frac{p_c(x,\exp_x(s_n \by))}{\omega_m R^m(0,s_n\by) f(x)} \ge 1-\alpha,
\]
which implies that
\[
 p_c(x,\exp_x(s_n \by)) \ge (1-\alpha)n^{-1} \omega_m R^m(0,\by) f(x).
\]
Therefore, we have
\begin{equation}\label{eq:p_lower_bound_super}
	e^{-np_c(x,\exp_x(s_n\by))} \le e^{-(1-\alpha)\omega_m R^m(0,\by) \fmin}.
\end{equation}
Finally, note that $R(0,\by) \ge \norm{y_i}/2$ for every $i$. Thus,
\[
	R^m(0, \by) \ge \frac{1}{2^m k}\sum_{i=1}^k \norm{y_i}^m.
\]
\end{itemize}
Overall, we have that the integrand in \eqref{eq:mean_nk_crit} is bounded by
\[
	\fmax^k f(x) e^{-\frac{(1-\alpha)\omega_m \fmin}{2^m k} \sum_{i=1}^k \norm{y_i}^m}.
\]
This function is integrable in $\cM\times(\Rm)^k$, and therefore we are done.
For the proof of the limit variance and CLT, see Appendix \ref{sec:apndx_moments} and \cite{bobrowski_distance_2011}.
\end{proof}

\begin{proof}[Proof of Proposition \ref{prop:lim_coverage}]

Since $\cM$ is $m$-dimensional, it can be shown that there exists $D>0$ such that for every $\eps$ we can find a (deterministic) set of points $\cS\subset \cM$ such that (a) $\cM \subset \cU(\cS, \eps)$, i.e.\! $\cS$ is $\eps$-dense  in $\cM$, and (b) $\abs{\cS} \le D \eps^{-m}$ (cf. \cite{flatto1977random}).

If $\cM$ is not covered by $\cU(\cP_n, r_n)$, then there exists $x\in\cM$, such that $\norm{x-X} > r_n$ for every $X\in\cP_n$.
For $\alpha >0$, let $\cS_n$ be a $( \alpha r_n)$-dense set in $\cM$, and let $s \in \cS_n$ be the closest point to $x$ in $\cS_n$. Then,
\[
\norm{x-X} \le \norm{x-s} + \norm{s-X}.
\]
Since $\norm{x-s} \le \alpha r_n$, then necessarily $\norm{s-X} > (1-\alpha)r_n$. Thus,
\begin{align*}
\prob{\cM \not\subset \cU(\cP_n, r_n)} \le \sum_{s\in \cS_n} \prob{B_{(1-\alpha)r_n}(s)\cap\cP_n = \emptyset} = \sum_{s\in\cS_n} e^{-nF(B_{(1-\alpha)r_n}(s))},
\end{align*}
where
\[
	F(B_{(1-\alpha)r_n}(s)) = \int_{B_{(1-\alpha)r_n}(s)\cap \cM} f(x)dx.
\]
Similarly to Lemma \ref{lem:lim_prob_vol} we can show that for every $x\in\cM$
\[
	\limninf \frac{F(B_{(1-\alpha)r_n}(x))}{\omega_m(1-\alpha)^mr_n^m} = f(x).
\]
Denoting
\[
G(x,\rho) = \begin{cases} \frac{F(B_{(1-\alpha)\rho})}{\omega_m(1-\alpha)^m\rho f(x)} & \rho > 0,\\
1 & \rho = 0, \end{cases}
\]
then $G:\cM\times[0,1]\to\R$ is continuous on a compact space, and therefore uniformly continuous. Thus, for every $\beta>0$ there exists $\tilde{\rho} >0$ such that for all $\rho < \tilde{rho}$ we have
$G(x,\rho) \ge 1-\beta$ for every $x\in\cM$. In other words, for $n$ large enough, we have that
\[
	F(B_{(1-\alpha)r_n}(x)) \ge (1-\beta)(1-\alpha)^m r_n^m\omega_m f(x),
\]
for every $x\in\cM$. Since $f(x) \ge \fmin > 0$, we have that,
\[
	\prob{\cM\not \subset \cU(\cP_n, r_n)} \le D(\alpha r_n)^{-m} e^{-(1-\alpha)^m(1-\beta)\fmin\omega_m n r_n^m}.
\]

We can now prove the two parts of the proposition.
\begin{enumerate}
\item
If we take $n\rnm \ge C \log n$ with $C \ge \frac{1}{(1-\alpha)^m(1-\beta)\fmin \omega_m}$, then we have
\[
\prob{\cM \not\subset \cU(\cP_n, r_n)} \le \tilde{D} \frac{1}{\log n} \to 0.
\]
Since we can choose $\alpha,\beta$ to be arbitrarily small, this statement holds for every $C >\frac{1}{\fmin \omega_m}$.

\item Similarly, if we take  $n\rnm \ge C \log n$ with $C \ge \frac{2+\eps}{(1-\alpha)^m(1-\beta)\fmin \omega_m}$, then we have
\[
\prob{\cM \not\subset \cU(\cP_n, r_n)} \le \tilde{D} \frac{1}{n^{(1+\eps)}\log n}.
\]
Therefore, we have that
\[
	\sum_{n=1}^{\infty} \prob{\cM \not \subset \cU(\cP_n,r_n))} < \infty,
\]
and from the Borel-Cantelli Lemma, we conclude that a.s.\! there exists $M>0$ such that for every $n>M$ we have $\cM \subset \cU(\cP_n,r_n)$.
\end{enumerate}
\end{proof}

To prove the result on $\Nkt$, we first prove the following lemma.
\begin{lem}\label{lem:mean_global}
For every $\eps > 0 $, if $C> \frac{1+\eps}{\fmin\omega_m}$, and $n\rnm \ge C \log n$, then there exists $D\ge 0$, such that
\[
	\mean{\Nkt-\Nk} \le D n^{-\eps}.
\]

\end{lem}
\begin{proof}
Similarly to the computation of $\Nk$, we have that
\[
    \mean{\Nkt} = \frac{n^{k+1}}{(k+1)!} \mean{h^c_{\rnh}(\cY) e^{-p_c(\cY)}}.
\]
Thus,
\begin{align*}
\mean{\Nkt - \Nk} &= \frac{n}{(k+1)!}\int_{\cM}\int_{(\Rm)^k} f(x,\exp_x(s_n \by))\\
&\times(h_{\rnh}^{c}(x,\exp_x(s_n\by))-h_{r_n}^{c}(x,\exp_x(s_n\by))) e^{-np_c(x,\exp_x(s_n\by))} d\by dx.
\end{align*}

Next, using Lemma \ref{lem:lim_prob_vol} we have that
\[
	\limninf \frac{p_c(x,\exp_x(s_n\by))}{\omega_m R^m(x,\exp_x(s_n\by))} = \limninf \frac{p_c(x,\exp_x(s_n\by))}{\omega_m s_n^m R(0,\by)} = f(x).
\]
We can use similar uniform continuity arguments to the ones used in the proof of Theorem \ref{thm:lim_crit}, to show that for a large enough $n$ we have that both
\begin{equation}\label{eq:p_bound_1}
	p_c(x,\exp_x(s_n\by)) \ge (1-\alpha)\omega_m R^m(x,\exp_x(s_n\by)) f(x),
\end{equation}
and
\begin{equation}\label{eq:p_bound_2}
	p_c(x,\exp_x(s_n\by)) \ge (1-\alpha)\omega_m s_n^m R^m(0,\by) f(x),
\end{equation}
for any $\alpha>0$. Now, if
\[
h_{\rnh}^{c}(x,\exp_x(s_n\by))- h_{r_n}^{c}(x,\exp_x(s_n\by))\ne 0,
\]
then necessarily $R(x, \exp_x(s_n\by)) \ge r_n$,
and from \eqref{eq:p_bound_1} we have that
\[
	p_c(x,\exp_x(s_n\by)) \ge (1-\alpha)\fmin \omega_m \rnm.
\]
Combining that with \eqref{eq:p_bound_2}, for every $\beta \in (0,1)$ we have that
\[
	np_c(x,\exp_x(s_n\by)) \ge \beta(1-\alpha)\fmin\omega_mR^m(0,\by) + (1-\beta)(1-\alpha)\fmin\omega_m n\rnm.
\]

Thus, we have that
\begin{align*}
&\mean{\Nkt - \Nk} \\
&\quad\le \frac{n e^{-(1-\alpha )(1-\beta)\omega_m \fmin n r_n^m}}{(k+1)!}\int_{\cM}\int_{(\Rm)^k}& \fmin^k f(x) e^{-\beta(1-\alpha)\fmin \omega_m R^m(0,\by)} d\by dx.
\end{align*}
The integral on the RHS is bounded. Thus, for any $\eps > 0$, if $C \ge \frac{1}{(1-\alpha)(1-\beta)}\frac{1+\eps}{\fmin \omega_m}$,  and $nr_n^m \ge C \log n$,  then
\[
	\mean{\Nkt-\Nk} \le D n^{-\eps}.
\]
This is true for any $\alpha,\beta>0$. Therefore, the statement holds for any $C>\frac{1+\eps}{\fmin \omega_m}$.
\end{proof}

\begin{proof}[Proof of Proposition \ref{prop:lim_glob}]
\begin{enumerate}
\item For every $1\le k \le m$,
\[
	\prob{\Nk \ne \Nkt} \le \mean{\Nkt-\Nk}.
\]
From Lemma \ref{lem:mean_global} we have that if $n\rnd \ge C\log n$ with $C>(\fmin \omega_m)^{-1}$ then
\[
	\limninf \prob{\Nk \ne \Nkt} = 0.
\]
Since
\[	
\prob{\exists k : \Nk \ne \Nkt } \le \sum_{k=1}^m \prob{\Nk \ne \Nkt} \to 0,
\]
we have that
\[
	\limninf \prob{\Nk = \Nkt,\ \forall 1\le k \le m} = 1.
\]

\item Next, if $C > 2(\fmin \omega_m)^{-1}$, then there exists $\eps > 0$ such that $2C > \frac{2+\eps}{\fmin \omega_m}$. Using Lemma \ref{lem:mean_global} we have that for $1\le k \le d$ there exists $D_k>0$ such that
\[
	\prob{\Nk \ne \Nkt} \le D_kn^{-(1+\eps)},
\]
Thus,
\[
	\sum_{n=1}^\infty \prob{\Nk \ne \Nkt} \le D_k\sum_{n=1}^\infty n^{-(1+\eps)}  < \infty.
\]
Using the Borel-Cantelli Lemma, we deduce that almost surely there exists $M_k>0$ (possibly random) such that for every $n>M_k$  we have
\[
	\Nk = \Nkt.
\]
Taking $M = \max_{1\le k \le m} M_k$, yields that for every $n>M$
\[
	\Nk = \Nkt,\quad \forall 1\le k \le m,
\]
which completes the proof.
\end{enumerate}
\end{proof}

\begin{proof}[Proof of Theorem \ref{thm:lim_betti_sup}]
If $n\rnm \ge C\log n$, and $C>(\omega_m \fmin)^{-1}$, then from Proposition \ref{prop:lim_coverage} we have that
\[
	\limninf\prob{\cM \subset \cU(\cP_n,r_n)} = 1.
\]
The deformation retract argument in \cite{niyogi_finding_2008} (Proposition 3.1) states that if $\cM \subset \cU(\cP_n, r_n)$, then $\cU(\cP_n, 2r_n)$ deformation retracts to $\cM$, and in particular - $\beta_k(\cU(\cP_n, 2r_n)) = \beta_k(\cM)$ for all $k$.
Thus, we have that
\begin{equation}\label{eq:lim_beta_2rn}
	\limninf\prob{\beta_k(\cU(\cP_n,2r_n)) = \beta_k(\cM)} = 1.
\end{equation}
Next,from Proposition \ref{prop:lim_glob} we have that
\[
	\limninf \prob{\Nk = \Nkt, \forall 1\le k \le m} = 1.
\]
By Morse theory, if $\Nk = \Nkt$ for every $k$, then necessarily $\beta_k(\cU(\cP_n,r_n)) = \beta_k(\cU(\cP_n,\hat{r}_n))$ for every $0\le k \le m$ (no critical points between $r_n$ and $\hat{r}_n$ implies no changes in the homology). Choosing $\hat{r}_n  = 2r_n$, we have that
\begin{equation}\label{eq:lim_rn_2rn}
	\limninf\prob{\beta_k(\cU(\cP_n, r_n)) = \beta_k(\cU(\cP_n, 2r_n))} = 1.
\end{equation}
Combining \eqref{eq:lim_beta_2rn} with \eqref{eq:lim_rn_2rn} yields
\[
	\limninf \prob{\bk = \beta_k(\cM),\ \forall 0\le k \le m} = 1,
\]
which completes the proof of the first part.
For the second part of the theorem , repeat the same arguments using the second part of propositions \ref{prop:lim_coverage} and \ref{prop:lim_glob}.
\end{proof} 

\section*{Acknowledgements}
The authors would like to thank: Robert Adler, Shmuel Weinberger, John Harer, Paul Bendich, Guillermo Sapiro, Matthew Kahle, Matthew Strom Borman,
and Alan Gelfand for many useful discussions. We would also like to thank the anonymous referee.

\appendix
\section{Palm Theory for Poisson Processes}

This appendix contains a collection of definitions and theorems which are used in the proofs of this paper. Most of the results are cited from \cite{penrose_random_2003}, although they may not necessarily have originated there. However, for notational reasons we refer the reader to \cite{penrose_random_2003}, while other resources include \cite{ arratia_two_1989, stoyan_stochastic_1987}.
The following theorem is very useful when computing expectations related to Poisson processes.

\begin{thm}[Palm theory for Poisson processes, \cite{penrose_random_2003}
Theorem 1.6]
\label{thm:palm}
Let $f$ be a probability density on $\R^d$, and let $\cP_n$ be a Poisson process on $\R^d$ with intensity $\lambda_n = n f$.
Let $h(\cY,\cX)$ be a measurable function defined for all finite subsets $\cY \subset \cX \subset \R^d$  with $\abs{\cY} = k$. Then
\[
    \E\Big\{\sum_{ \cY \subset \cP_n}
    h(\cY,\cP_n)\Big\} = \frac{n^k}{k!} \mean{h(\cY',\cY' \cup \cP_n)}
\]
where $\cY'$ is a set of $k$ $iid$ points in $\Rd$ with density $f$, independent of $\cP_n$.
\end{thm}
We shall also need the following corollary, which treats second moments:
\begin{cor}\label{cor:palm2}
With the notation above, assuming $\abs{\cY_1} = \abs{\cY_2} = k$,
\[
    \E\Big\{\sum_{ \substack {
                    \cY_1 ,\cY_2\subset \cP_n  \\
                    \abs{\cY_1 \cap \cY_2} = j }}
    h(\cY_1,\cP_n)h(\cY_2,\cP_n)\Big\} = {\frac{n^{2k-j}}{j!((k-j)!)^2}} \mean{h(\cY_1',\cY_{12}' \cup \cP_n)h(\cY_2',\cY_{12}' \cup \cP_n)}
\]
where $\cY'_{12} = \cY'_1 \cup \cY'_2$ is a set of $2k-j$ $iid$ points in $\Rd$ with density $f(x)$, independent of $\cP_n$, and $\abs{\cY_1'\cap\cY_2'} = j$.
\end{cor}

\section{Stein's Method}\label{sec:apndx_stein}
In this paper we omitted the proofs for the limit distributions in Theorems \ref{thm:dist_subcrit}, \ref{thm:lim_crit}, and \ref{thm:lim_supcrit},  referring the reader to \cite{bobrowski_distance_2011}, where these results were proved for point processes in a Euclidean space. These proof mainly rely on moment computations similar to the ones presented in this paper, but technically more complicated. In this section we wish to introduce the main theorems used in these proofs. 

The theorems below are two instances of  \emph{Stein's method}, used to prove limit distribution for sums of weakly dependent variables. To adapt these method to the statements in this paper, one can think of the random variables $\xi_i$ as some version of the Bernoulli variables $\grn^b(\cY,\cP_n), \grn^c(\cY,\cP_n)$ used in this paper.

\begin{defn}\label{def:dep_graph}
Let $(I,E)$ be a graph. For $i,j\in I$ we denote $i \sim j$ if $(i,j) \in E$. Let $\set{\xi_i}_{i\in I}$ be a set of random variables. We say that $(I,\sim)$ is a dependency graph for $\set{\xi_i}$ if for every $I_1\cap I_2 = \emptyset$, with no edges between $I_1$ and $I_2$, the set of variables $\set{\xi_i}_{i\in I_1}$ is independent of $\set{\xi_i}_{i\in I_2}$. We also define the neighborhood of $i$ as $\cN_i := \set{i} \cup \set{j \in I \: j\sim i}$.
\end{defn}

\begin{thm}[Stein-Chen Method for Bernoulli Variables, Theorem 2.1 in \cite{penrose_random_2003}] \label{thm:prelim:stein_bern}
Let $\set{\xi_i}_{i\in I}$ be a set of Bernoulli random variables, with dependency graph $(I,\sim)$. Let
\[
p_i := \mean{\xi_i},  \ \ p_{i,j} := \mean{\xi_i \xi_j}, \ \
 \lambda := \sum_{i\in I} p_i, \ \ W:= \sum_{i\in I} \xi_i,
\ \ Z\sim\pois{\lambda}.
\]
Then,
\[
\dtv{W}{Z} \le \min(3,\lambda^{-1}) \Big(\sum_{i\in I}\sum_{j\in \cN_i\bs \set{i}} p_{ij} + \sum_{i\in I}\sum_{j \in \cN_i} p_i p_j\Big).
\]
\end{thm}
\begin{thm}[CLT for sums of weakly dependent variables, Theorem 2.4 in  \cite{penrose_random_2003}]\label{thm:clt_stein}
Let $(\xi_i)_{i\in I}$ be a finite collection of random variables, with $\mean{\xi_i} = 0$. Let $(I,\sim)$ be the dependency graph of $(\xi_i)_{i\in I}$, and assume that its maximal  degree is $D-1$. Set $W:=\sum_{i\in I} \xi_i$, and suppose that $\mean{W^2}=1$. Then for all $w\in \R$,
\[
\abs{F_W(w) - \Phi(w)} \le 2(2\pi)^{-1/4} \sqrt{D^2 \sum_{i\in I} \mean{\abs{\xi_i}^3}} + 6\sqrt{D^3 \sum_{i\in I} \mean{\abs{\xi_i}^4}},
\]
where $F_W$ is the  distribution function of $W$ and $\Phi$ that of a standard Gaussian.
\end{thm}

\section{Second Moment Computations}
\label{sec:apndx_moments}
In this section we briefly review the steps required to evaluate the second moment of either $\bk$ or $\Nk$ in order to compute the limit variance in Theorems \ref{thm:dist_subcrit}, \ref{thm:lim_crit}, and \ref{thm:lim_supcrit}. Similar computations are required to evaluate higher moments, which are needed in order to apply Stein's method for the limit distributions. The proofs follow the same steps as the proofs in both \cite{kahle_random_2011} and \cite{bobrowski_distance_2011}. These proofs are long and technically complicated, and since repeating them again for the manifold case should add no insight, we refer the reader to these papers for the complete proofs.

We present the statements in terms of $\Nk$, but the same line of arguments can be  applied to $\Sk$ as well (defined in \ref{eq:def_sk}).

The variance of $\Nk$ is 
\begin{equation}\label{eq:var_nk}
\var{\Nk}= \meanx{\Nk^2} - (\mean{N_k})^2.
\end{equation}
The first term on the right hand side can be written as
\begin{align}
\mean{\Nk^2} &= \mean{\sum_{\cY_1\subset\cP_n}\sum_{\cY_2\subset\cP_n} {\grn(\cY_1,\cP_n)\grn(\cY_2, \cP_n)}}
\notag
\\
&=\sum_{j=0}^{k+1}\mean{\sum_{\cY_1\subset\cP_n}\sum_{\cY_2\subset\cP_n}{ \grn(\cY_1,\cP_n)\grn(\cY_2,\cP_n)}\ind\set{\abs{\cY_1\cap\cY_2}=j}} \notag
\\
&:= \sum_{j=0}^{k+1} \mean{{I}_j}. \label{eq:def_I_j}
\end{align}

Note that
\begin{equation}\label{eq:mean_I_k1}
I_{k+1} = \sum_{ \cY_1 \subset \cP_n }
    \grn(\cY_1,\cP_n) =\Nk,
\end{equation}
and we know the limit of the expectation of this term in each of the regimes.

Next, for $0\le j<k+1$, using Corollary \ref{cor:palm2} we have
\begin{equation}\label{eq:I_j}
\mean{I_j} = \frac{n^{2k+2-j}}{j!((k+1-j)!)^2}\mean{ \grn(\cY_1',\cY_{12}'\cup\cP_n)\grn(\cY_2',\cY_{12}'\cup\cP_n)},
\end{equation}
where $\cY'_{12} = \cY'_1 \cup \cY'_2$ is a set of $(2k-j)$ $iid$ points in $\Rd$ with density $f(x)$, independent of $\cP_n$,  $\abs{\cY_1'} = \abs{\cY_2'} = k$, and $\abs{\cY_1'\cap\cY_2'} = j$. For $j>0$, the functional inside the expectation is nonzero for subsets $\cY_{1,2} '$ contained in a ball of radius $4r_n$. Thus, a change of variables similar to the ones used in the proof of Theorems \ref{thm:dist_subcrit}, \ref{thm:lim_crit} and \ref{thm:lim_supcrit}, can be used to show that this expectation on the right hand side of \eqref{eq:I_j} is $O(r_n^{m(2k+1-j)})$. If $j=0$ the sets are disjoint, and given $\cY_1'$ and $\cY_2'$ we have two options: If $B(\cY_1') \cap B(\cY_2') \ne \emptyset$, then a similar bound to the on above applies. Otherwise, the two balls are disjoint, and therefore the processes $B(\cY_1') \cap \cP_n$ and $B(\cY_1') \cap \cP_n$ are independent. In this case it can be show that the expected value cancels with $\meanx{\Nk^2}$ in \eqref{eq:var_nk}.

In the subcritical regime, the dominated term in \eqref{eq:def_I_j} would be $\mean{I_{k+1}}$, and from \eqref{eq:mean_I_k1} we have that $\var{N_k} \approx \mean{N_k}$. In the other regimes, all the terms in \eqref{eq:def_I_j} are $O(n)$, and thus the limit variance is $O(n)$ as well.


\bibliographystyle{plain}
\bibliography{refs}

\begin{thebibliography}{10}

\bibitem{adler_persistent_2010}
Robert~J. Adler, Omer Bobrowski, Matthew~S. Borman, Eliran Subag, and Shmuel
  Weinberger.
\newblock Persistent homology for random fields and complexes.
\newblock {\em Institute of Mathematical Statistics Collections}, 6:124--143,
  2010.

\bibitem{ABW}
Robert~J. Adler, Omer Bobrowski, and Shmuel Weinberger.
\newblock Crackle: The persistent homology of noise, 2013.
\newblock http://arxiv.org/abs/1301.1466.

\bibitem{adler_random_2007}
Robert~J. Adler and Jonathan~E. Taylor.
\newblock {\em Random fields and geometry}.
\newblock Springer Monographs in Mathematics. Springer, New York, 2007.

\bibitem{AswaniBickelTomlin}
Peter~Bickel Anil~Aswani and Claire Tomlin.
\newblock Regression on manifolds: Estimation of the exterior derivative.
\newblock {\em Annals of Statistics}, 39(1):48--81, 2011.

\bibitem{aronshtam_vanishing_2010}
Lior Aronshtam, Nathan Linial, Tomasz Luczak, and Roy Meshulam.
\newblock Vanishing of the top homology of a random complex.
\newblock {\em Arxiv preprint {arXiv:1010.1400}}, 2010.

\bibitem{arratia_two_1989}
Richard Arratia, Larry Goldstein, and Louis Gordon.
\newblock Two moments suffice for poisson approximations: the {{C}hen-{S}tein}
  method.
\newblock {\em The Annals of Probability}, 17(1):9--25, 1989.

\bibitem{AuffingerBenArous}
Antonio Auffinger and G\'{e}rard~Ben Arous.
\newblock Complexity of random smooth functions of many variables.
\newblock {\em Annals of Probability}, 2013.

\bibitem{BS1984}
A.J. Baddeley and B.W. Silverman.
\newblock A cautionary example on the use of second-order methods for analyzing
  point patterns.
\newblock {\em Biometrics}, 40:1089--1094, 1984.

\bibitem{baryshnikov2013min}
Yuliy Baryshnikov, Peter Bubenik, and Matthew Kahle.
\newblock Min-type morse theory for configuration spaces of hard spheres.
\newblock {\em International Mathematics Research Notices}, page rnt012, 2013.

\bibitem{BelkinNiyLap}
Mikhail Belkin and Partha Niyogi.
\newblock Towards a theoretical foundation for laplacian-based manifold
  methods.
\newblock In Peter Auer and Ron Meir, editors, {\em Learning Theory}, volume
  3559 of {\em Lecture Notes in Computer Science}, pages 486--500. Springer
  Berlin Heidelberg, 2005.

\bibitem{BenMukWang2012}
P.~Bendich, S.~Mukherjee, and B.~Wang.
\newblock Local homology transfer and stratification learning.
\newblock {\em ACM-SIAM Symposium on Discrete Algorithms}, 2012.

\bibitem{bobrowski2012thesis}
Omer Bobrowski.
\newblock Algebraic topology of random fields and complexes.
\newblock {\em PhD Thesis}, 2012.

\bibitem{bobrowski_distance_2011}
Omer Bobrowski and Robert~J. Adler.
\newblock Distance functions, critical points, and topology for some random
  complexes.
\newblock {\em {arXiv:1107.4775}}, July 2011.

\bibitem{bobrowski2012euler}
Omer Bobrowski and Matthew~Strom Borman.
\newblock Euler integration of {G}aussian random fields and persistent
  homology.
\newblock {\em Journal of Topology and Analysis}, 4(01):49--70, 2012.

\bibitem{borsuk_imbedding_1948}
Karol Borsuk.
\newblock On the imbedding of systems of compacta in simplicial complexes.
\newblock {\em Fund. Math}, 35(217-234):5, 1948.

\bibitem{Bryzgalova78}
L.N. Bryzgalova.
\newblock The maximum functions of a family of functions that depend on
  parameters.
\newblock {\em Funktsional. Anal. i Prilozhen}, 12(1):66--67, 1978.

\bibitem{bubenik_statistical_2009}
Peter Bubenik, Gunnar Carlsson, Peter~T. Kim, and Zhiming Luo.
\newblock Statistical topology via {{M}orse} theory, persistence and
  nonparametric estimation.
\newblock {\em 0908.3668}, August 2009.
\newblock Contemporary Mathematics, Vol. 516 (2010), pp. 75-92.

\bibitem{bubenik_statistical_2007}
Peter Bubenik and Peter~T. Kim.
\newblock A statistical approach to persistent homology.
\newblock {\em Homology, Homotopy and Applications}, 9(2):337--362, 2007.

\bibitem{TaylorWorsley2}
N.~Chamandy, K.J. Worsley, J.E. Taylor, and F.~Gosselin.
\newblock Tilted {E}uler characteristic densities for central limit random
  fields, with applications to "bubbles".
\newblock {\em Annals of Statistics}, 36(5):2471--2507, 2008.

\bibitem{ChaCohLie2009}
F.~Chazal, D.~Cohen-Steiner, and A.~Lieutier.
\newblock A sampling theory for compact sets in {E}uclidean space.
\newblock {\em Discrete and Computational Geometry}, 41:461--479, 2009.

\bibitem{ChenMuller}
Dong Chen and Hans-Georg M\"{u}ller.
\newblock Nonlinear manifold representations for functional data.
\newblock {\em Annals of Statistics}, 40(1):1--29, 2012.

\bibitem{Genoveseetal}
Isabella~Verdinelli Christopher R.~Genovese, Marco Perone-Pacifico and Larry
  Wasserman.
\newblock On the path density of a gradient field.
\newblock {\em Annals of Statistics}, 37(6A):3236--3271, 2009.

\bibitem{chung_persistence_2009}
Moo~K. Chung, Peter Bubenik, and Peter~T. Kim.
\newblock Persistence diagrams of cortical surface data.
\newblock In {\em Information Processing in Medical Imaging}, page 386–397,
  2009.

\bibitem{Diggle2003}
Peter~J. Diggle.
\newblock {\em Statistical Analysis of Spatial Point Patterns}.
\newblock Academic Press, 2003.

\bibitem{flatto1977random}
Leopold Flatto and Donald~J Newman.
\newblock Random coverings.
\newblock {\em Acta Mathematica}, 138(1):241--264, 1977.

\bibitem{gershkovich_morse_1997}
Vladimir Gershkovich and Hyam Rubinstein.
\newblock Morse theory for min-type functions.
\newblock {\em The Asian Journal of Mathematics}, 1(4):696--715, 1997.

\bibitem{hatcher_algebraic_2002}
Allen Hatcher.
\newblock {\em Algebraic topology}.
\newblock Cambridge University Press, Cambridge, 2002.

\bibitem{kahle_topology_2009}
Matthew Kahle.
\newblock Topology of random clique complexes.
\newblock {\em Discrete Mathematics}, 309(6):1658--1671, 2009.

\bibitem{kahle_random_2011}
Matthew Kahle.
\newblock Random geometric complexes.
\newblock {\em Discrete \& Computational Geometry. An International Journal of
  Mathematics and Computer Science}, 45(3):553--573, 2011.

\bibitem{kahle_limit_2010}
Matthew Kahle, Elizabeth Meckes, et~al.
\newblock Limit the theorems for {B}etti numbers of random simplicial
  complexes.
\newblock {\em Homology, Homotopy and Applications}, 15(1):343--374, 2013.

\bibitem{RoheChatterjeeYu}
Sourav~Chatterjee Karl~Rohe and Bin Yu.
\newblock Spectral clustering and the high-dimensional stochastic blockmodel.
\newblock {\em Annals of Statistics}, 39(4):1878--1915, 2011.

\bibitem{linial2006homological}
Nathan Linial and Roy Meshulam.
\newblock Homological connectivity of random 2-complexes.
\newblock {\em Combinatorica}, 26(4):475--487, 2006.

\bibitem{Luna:2009}
S.~Lunag\'omez, S.~Mukherjee, and Robert~L. Wolpert.
\newblock Geometric representations of hypergraphs for prior specification and
  posterior sampling, 2009.
\newblock http://arxiv.org/abs/0912.3648.

\bibitem{Matheron1975}
G.~Matheron.
\newblock {\em Random sets and integral geometry}.
\newblock John Wiley\thinspace \&\thinspace Sons, New York-London-Sydney, 1975.
\newblock With a foreword by Geoffrey S. Watson, Wiley Series in Probability
  and Mathematical Statistics.

\bibitem{Matov82}
V.I. Matov.
\newblock Topological classication of the germs of functions of the maximum and
  minimax of families of functions in general position.
\newblock {\em Uspekhi Mat. Nauk}, 37(4(226)):167--168, 1982.

\bibitem{ms05}
Klaus~R. Mecke and Dietrich Stoyan.
\newblock Morphological characterization of point patterns.
\newblock {\em Biometrical Journal}, 47(5):473--488, 2005.

\bibitem{MeesterRoy96}
R.~Meester and R.~Roy.
\newblock {\em Continuum percolation}.
\newblock Cambridge University Press, 1996.

\bibitem{memoli2005distance}
Facundo M{\'e}moli and Guillermo Sapiro.
\newblock Distance functions and geodesics on submanifolds of $\mathbb{R}^d$
  and point clouds.
\newblock {\em SIAM Journal on Applied Mathematics}, 65(4):1227--1260, 2005.

\bibitem{milnor_morse_1963}
John~W. Milnor.
\newblock {\em Morse theory}.
\newblock Based on lecture notes by M. Spivak and R. Wells. Annals of
  Mathematics Studies, No. 51. Princeton University Press, Princeton, {N.J.},
  1963.

\bibitem{MischaikowWanner}
Konstantin Mischaikow and Thomas Wanner.
\newblock Probabilistic validation of homology computations for nodal domains.
\newblock {\em Annals of Applied Probability}, 17(3):980--1018, 2007.

\bibitem{Molchanov2005}
I.~Molchanov.
\newblock {\em Theory of random sets}.
\newblock Springer., 2005.

\bibitem{MW2003}
J.~Moller and R.~Waagepetersen.
\newblock {\em Statistical Inference for Spatial Point Processes}.
\newblock Chapman \& Hall, 2003.

\bibitem{munkres1984elements}
James~R Munkres.
\newblock {\em Elements of algebraic topology}, volume~2.
\newblock Addison-Wesley Reading, 1984.

\bibitem{niyogi_finding_2008}
Partha Niyogi, Stephen Smale, and Shmuel Weinberger.
\newblock Finding the homology of submanifolds with high confidence from random
  samples.
\newblock {\em Discrete \& Computational Geometry. An International Journal of
  Mathematics and Computer Science}, 39(1-3):419--441, 2008.

\bibitem{niyogi_topological_2011}
Partha Niyogi, Stephen Smale, and Shmuel Weinberger.
\newblock A topological view of unsupervised learning from noisy data.
\newblock {\em {SIAM} Journal on Computing}, 40(3):646, 2011.

\bibitem{penrose_random_2003}
Mathew~D. Penrose.
\newblock {\em Random geometric graphs}, volume~5 of {\em Oxford Studies in
  Probability}.
\newblock Oxford University Press, Oxford, 2003.

\bibitem{penrose_limit_2011}
Mathew~D. Penrose and Joseph~E. Yukich.
\newblock Limit theory for point processes in manifolds.
\newblock {\em 1104.0914}, April 2011.

\bibitem{Ripley}
B.~D. Ripley.
\newblock The second-order analysis of stationary point processes.
\newblock {\em Annals of Applied Probability}, 13(2):255--266, 1976.

\bibitem{Adleretal}
Eliran~Subag Robert J.~Adler and Jonathan~E. Taylor.
\newblock Rotation and scale space random fields and the gaussian kinematic
  formula.
\newblock {\em Annals of Statistics}, 40(6):2910--2942, 2012.

\bibitem{WuMukZhou}
Qiang~Wu Sayan~Mukherjee and Ding-Xuan Zhou.
\newblock Learning gradients on manifolds.
\newblock {\em Bernoulli}, 16(1):181--207, 2010.

\bibitem{stoyan_stochastic_1987}
Dietrich Stoyan, Wilfried~S. Kendall, and Joseph Mecke.
\newblock {\em Stochastic geometry and its applications}.
\newblock Wiley Series in Probability and Mathematical Statistics: Applied
  Probability and Statistics. John Wiley \& Sons Ltd., Chichester, 1987.
\newblock With a foreword by D. G. Kendall.

\bibitem{TaylorWorsley1}
J.E. Taylor and K.J. Worsley.
\newblock Random fields of multivariate test statistics, with applications to
  shape analysis.
\newblock {\em Annals of Statistics}, 36(1):1--27, 2008.

\bibitem{worsley_boundary_1995}
Keith~J. Worsley.
\newblock Boundary corrections for the expected euler characteristic of
  excursion sets of random fields, with an application to astrophysics.
\newblock {\em Advances in Applied Probability}, pages 943--959, 1995.

\bibitem{worsley_estimating_1995}
Keith~J. Worsley.
\newblock Estimating the number of peaks in a random field using the
  {{H}adwiger} characteristic of excursion sets, with applications to medical
  images.
\newblock {\em The Annals of Statistics}, 23(2):640--669, April 1995.
\newblock Mathematical Reviews number {(MathSciNet):} {MR1332586;} Zentralblatt
  {MATH} identifier: 0898.62120.

\end{thebibliography}
\end{document}